\documentclass[11pt,oneside]{amsart}

\usepackage[utf8]{inputenc}
\usepackage[english]{babel}
\usepackage[margin=2cm, right=4cm]{geometry} 
\usepackage{amsmath,amsthm,amssymb}
\usepackage[normalem]{ulem}
\usepackage{mathtools}
\usepackage{graphicx}
\usepackage[colorlinks=true, allcolors=blue]{hyperref}
\usepackage{mathrsfs}
\usepackage{tikz-cd,adjustbox}
\usepackage{spverbatim,comment}
\usepackage[shortlabels]{enumitem}
\usepackage[capitalize]{cleveref}
\usepackage{adjustbox}
\usepackage{quiver}

\usepackage[alphabetic,initials]{amsrefs}

\usepackage[textsize=tiny]{todonotes}

\theoremstyle{plain}
\newtheorem{theorem}{Theorem}[section]
\newtheorem{lemma}[theorem]{Lemma}
\newtheorem{proposition}[theorem]{Proposition}
\newtheorem{corollary}[theorem]{Corollary}

\newtheorem{construction}[theorem]{Construction}
\newtheorem{situation}[theorem]{Situation}

\theoremstyle{definition}
\newtheorem{definition}[theorem]{Definition}
\newtheorem{example}[theorem]{Example}

\theoremstyle{remark}
\newtheorem{remark}[theorem]{Remark}

\newcommand{\Rom}[1]{\uppercase\expandafter{\romannumeral#1}}

\newcommand{\A}{\mathcal A}

\newcommand{\E}{\mathcal E}
\newcommand{\F}{\mathcal F}

\renewcommand{\O}{\mathcal O}

\newcommand{\ZZ}{\mathbb Z}

\newcommand{\fil}{\mathrm{Fil}}
\newcommand{\gr}{\mathrm{gr}}

\DeclareMathOperator{\Tr}{Tr}

\newcommand{\GL}{\operatorname{GL}}

\DeclareMathOperator{\Fil}{Fil}

\DeclareMathOperator{\im}{Im}
\DeclareMathOperator{\Lotimes}{\otimes^{\bf{L}}}
\DeclareMathOperator{\Mod}{Mod}
\DeclareMathOperator{\rank}{rank}

\DeclareMathOperator{\Tot}{Tot}

\newcommand{\blank}{\underline{\,\,\,\,}} 
\DeclareMathOperator{\Hom}{Hom}

\DeclareMathOperator{\localRHom}{R \mathcal{H} om}
\DeclareMathOperator{\localHom}{\mathcal{H} om}
\DeclareMathOperator{\Ext}{Ext}
\DeclareMathOperator{\Tor}{Tor}

\DeclareMathOperator{\CF}{CF_b}
\DeclareMathOperator{\CFD}{CFD_b}
\DeclareMathOperator{\DF}{DF_b}
\DeclareMathOperator{\DFD}{DFD_b}
\DeclareMathOperator{\FAc}{FAc_b}
\DeclareMathOperator{\FDAc}{FDAc_b}
\DeclareMathOperator{\FDQis}{FDQis_b}
\DeclareMathOperator{\FQis}{FQis_b}
\DeclareMathOperator{\KF}{KF_b}
\DeclareMathOperator{\KFD}{KFD_b}

\DeclareMathOperator{\Spec}{Spec}

\DeclareMathOperator{\id}{id}
\DeclareMathOperator{\Lie}{Lie}
\DeclareMathOperator{\QCoh}{QCoh}

\DeclarePairedDelimiter\abs{\lvert}{\rvert}

\makeatletter
\let\oldabs\abs
\def\abs{\@ifstar{\oldabs}{\oldabs*}}
\makeatother

\theoremstyle{definition} 
\newcommand{\thistheoremname}{}
\newtheorem*{genericthm}{\thistheoremname}

  \newcommand{\mb}{\mathbb}
\newcommand{\mc}{\mathcal}

\renewcommand{\O}{\mc{O}}

\newcommand{\eps}{\epsilon}
\newcommand*{\sheafhom}{\mathcal{H}\kern -.5pt om}

\newcommand{\be}{\begin{equation*}}
\newcommand{\ee}{\end{equation*}}
\newcommand{\bel}{\begin{equation}}
\newcommand{\eel}{\end{equation}}
\newcommand{\bea}{\begin{eqnarray}}
\newcommand{\eea}{\end{eqnarray}}
\newcommand{\ben}{\begin{enumerate}}
\newcommand{\een}{\end{enumerate}}
\newcommand{\bi}{\begin{itemize}}
\newcommand{\ei}{\end{itemize}}

\newcommand{\gw}{\omega}
\newcommand{\ep}{\epsilon}

\newcommand{\bt}{\boxtimes}
\newcommand{\tten}{\otimes^{\mb{L}}}
\newcommand{\ebe}{\mathcal{E}\boxtimes \mathcal{E}^*}
\newcommand{\rg}{R\Gamma}

\newcommand{\cA}{\mathcal{A}}

\newcommand{\cC}{\mathcal{C}}

\newcommand{\cE}{\mathcal{E}}
\newcommand{\cF}{\mathcal{F}}
\newcommand{\cG}{\mathcal{G}}
\newcommand{\cH}{\mathcal{H}}

\newcommand{\cO}{\mathcal{O}}
\newcommand{\cP}{\mathcal{P}}

\newcommand{\cS}{\mathcal{S}}

\newcommand{\dd}[1]{\frac{\partial}{\partial #1}}
\newcommand{\ddd}[2]{\frac{\partial #1}{\partial #2}}

\title{Duality of differential operators and algebraic de Rham cohomology}

\author{Caleb Ji}
\address{Caleb Ji \\
    Department of Mathematics \\
	Columbia University \\
	New York, NY 10027 }
\email{cj2670@columbia.edu}

\author{Casimir Kothari}
\address{Casimir Kothari \\
    Department of Mathematics \\
	University of Chicago \\
	Chicago, IL 60637 }
\email{ckothari@uchicago.edu}

\author{Oliver Li}
\address{Oliver Li \\
    Department of Mathematics \\
	University of Melbourne \\
	Parkville VIC 3052, Australia }
\email{oliver.li@student.unimelb.edu.au}

\author{Svetlana Makarova}
\address{Svetlana Makarova \\
    Department of Mathematics \\
	Australian National University \\
	Canberra ACT 2601, Australia }
\email{svetlana.makarova@anu.edu.au}

\author{Shubhankar Sahai}
\address{Shubhankar Sahai \\
    Department of Mathematics \\
	University of California San Diego \\
	La Jolla, CA 92093 }
\email{ssahai@ucsd.edu}

\author{Sridhar Venkatesh}
\address{Sridhar Venkatesh \\
    Department of Mathematics \\
	University of Michigan \\
	Ann Arbor, MI 48109}
\email{srivenk@umich.edu}

\date{\today}

\begin{document}

\begin{abstract}
    Given a smooth proper morphism $f\colon X\rightarrow S$, we introduce a certain derived category where morphisms are permitted to be $\O_S$-linear differential operators. 
 We then prove a generalisation of Serre duality that applies to two-term complexes of this type.  We apply this to give a new proof of Poincar\'e duality for relative algebraic de Rham cohomology. 
\end{abstract}

\maketitle

\setcounter{tocdepth}{1}
\tableofcontents

\section{Introduction} 
Let $X$ be a smooth proper variety of dimension $n$ over a field of characteristic 0. 
Poincar\'e duality for algebraic de Rham cohomology for $X$, proven by Hartshorne in \cite{Hartshorne}*{Corollary II.5.3}, gives an isomorphism 
\[
H^k(X, \Omega_{X}^{\bullet})\cong H^{2n-k}(X, \Omega_{X}^{\bullet})^\vee. 
\]
Recalling the fact that for such schemes, the dualising complex is isomorphic to $\Omega^n_X$,  one can notice that this statement is reminiscent of Serre duality.
However, its direct application is prevented by the fact that the differentials occurring in the algebraic de Rham complex are not $\O_X$-linear.  To resolve this issue, Hartshorne shows in his proof that Serre duality is compatible with the differentials in a suitable sense. 

We now know that Poincar\'e duality for algebraic de Rham cohomology holds in much greater generality.  In fact, a relative version holds for any quasi-compact quasi-separated base. 

\begin{theorem}
\cite[\href{https://stacks.math.columbia.edu/tag/0G8K}{Tag 0G8K}]{stacks-project} 
Let $S$ be a quasi-compact and quasi-separated scheme. Let $f \colon X \to S$
be a proper smooth morphism of schemes all of whose fibres are nonempty
and equidimensional of dimension $n$. Then there exists an
$\mathcal{O}_S$-module map
$$
t \colon R^{2n}f_*\Omega^\bullet_{X/S} \longrightarrow \mathcal{O}_S,
$$
unique up to precomposing by multiplication by a unit of
$H^0(X, \mathcal{O}_X)$,
such that the pairing
$$
Rf_*\Omega^\bullet_{X/S}
\otimes_{\mathcal{O}_S}^\mathbb{L}
Rf_*\Omega^\bullet_{X/S}[2n]
\longrightarrow
\mathcal{O}_S, \quad
(\xi, \xi') \longmapsto t(\xi \cup \xi')
$$
is a perfect pairing of perfect complexes on $S$.
\end{theorem} 

The proof in the Stacks project uses the Hodge-de Rham spectral sequence, reducing the statement to the vanishing of certain differentials.  The key point is the following lemma. 

\begin{lemma}
\cite[\href{https://stacks.math.columbia.edu/tag/0G8J}{Tag 0G8J}]{stacks-project}
\label{lemma: key}
    Let $S$ be a quasi-compact and quasi-separated scheme and $f\colon X\rightarrow S$ be a proper smooth morphism of schemes all of whose fibres are nonempty and equidimensional of dimension $n$.  Then the map $d^{n-1} \colon R^nf_*\Omega^{n-1}_{X/S}\rightarrow R^nf_*\Omega^n_{X/S}$ is zero.
\end{lemma}  

The currently existing proof of this lemma in the Stacks Project uses intricate computations with explicit Gysin maps
 and hides the fact that $d^{n-1}$ turns out to be the dual of another differential, which happens to be $0$ for trivial reasons.
Indeed, we recall that the map $d^0 \colon f_*\O_{X}\rightarrow f_*\Omega^1_{X/S}$ is $0$ \cite[\href{https://stacks.math.columbia.edu/tag/0G8F}{Tag 0G8H}]{stacks-project}. 
Clausen proposed another proof in \cite{clausen} which, though sharing some of the characteristics of the proof we will give, also does go through this duality statement. 
De Jong proposed to prove an extension of Serre duality to the case of differential operators, which would show that the differentials $d^{n-1}$ and $d^0$ are dual, thus simplifying the existing proof of the lemma.

Denote $\omega_{X/S}\coloneqq \Omega^n_{X/S}$.  If $\mc{E}$ is a locally free $\O_X$-module of finite rank, define its dual $\mathcal{E}^*\coloneqq \mc{E}^\vee\otimes_{\O_S}\omega_{X/S}$, where $\mc E^\vee := \localHom (\mc E , \mc O_X)$.  
Given a differential operator $D\colon \mc{E}\rightarrow\mc{F}$ between two locally free $\O_X$-modules, we will first define a dual differential operator $D^* \colon \mc{F}^*\rightarrow\mathcal{E}^*$ that satisfies $(d^0)^*=-(d^{n-1})^*$.
This definition naturally extends to the filtered derived category of $S$-differential complexes with uniformly bounded filtrations, a category which we define and study in detail in Section \ref{sec: kunneth}.  
The main new result of this paper is Theorem \ref{thm: copairing}, a duality statement that extends Serre duality to take differential operators into account.  
More concretely, it implies the following statement.

\begin{proposition} 
\label{prop: serre d}
Let $f\colon X\rightarrow S$ be smooth and proper of relative dimension $n$, where $S=\Spec A$ is affine. Let $M$ be a two-term complex of locally free sheaves of finite rank in the filtered derived category of $S$-differential complexes with uniformly bounded filtrations
    Then there is a perfect pairing of complexes 
    \[
Rf_*M \otimes^{\mb{L}}_{\O_S} Rf_* (M^*)[n] \rightarrow \O_S.  
    \]
\end{proposition}

We will use Theorem \ref{thm: copairing} to prove Lemma \ref{lemma: key} by taking $M$ to be the complex $\O_X\xrightarrow{d}\Omega^1_{X/S}$, thus giving a new proof of Poincar\'e duality for algebraic de Rham cohomology in the relative setting.  

\subsection{Outline of paper}
We will now outline the contents of this paper in more detail.  
In Section \ref{sec: dual diff}, we construct the dual differential operator $D^*\colon \mc{F}^*\rightarrow \mc{E}^*$ to a differential operator $D\colon \mc{E}\rightarrow \mc{F}$ between locally free sheaves of finite rank.
Over the next few sections, given a two-term complex of the form $M=\mc{E}\xrightarrow{D}\mc{F}$, our goal is to construct a perfect copairing $\eta\colon A\rightarrow R\Gamma(X, M)[n]\otimes_A^{\mathbb{L}}R\Gamma(X, M^*)$.    
In Section \ref{sec: kunneth}, we will prove a version of the filtered K\"unneth isomorphism so that such a copairing is equivalent to giving an element $\eta(1)\in H^n(X\times_S X, M\boxtimes M^*)$.   
In Section \ref{sec: copairing locally free} we will construct a sequence of maps 
\[
H^0(\Delta, \mathcal{H}^n_{\Delta}(\mathcal{E})) \cong H^n_{\Delta}(X\times_S X, \mc{E}\boxtimes \mc{E}^*) \rightarrow H^n(X\times_S X, \mc{E}\boxtimes \mc{E}^*), 
\] 
and we will produce an element $\eta(1)$ by explicitly writing down a global section of $\mathcal{H}^n_{\Delta}(\mc{E})$ on $\Delta$.  In Section \ref{sec: compatibility} we will confirm that this element produces a copairing compatible with Serre duality.  In Section \ref{sec: two-term}, we will extend our constructions and results of the previous two sections to two-term complexes, proving Theorem \ref{thm: copairing}.  Finally in Section \ref{sec: duality diff}, we will give the full argument that Theorem \ref{thm: copairing} implies Lemma \ref{lemma: key}.

\subsection{Acknowledgements} 
We would like to thank Johan de Jong, who proposed this project to us and spent many hours explaining it to us and answering our questions about it.  
We would also like to thank all the 2023 Stacks Project Workshop organisers for bringing us together to this workshop where this work began. 
We were supported by the grants NSF DMS-2309115 and DMS-1840234 at the workshop.  
Additionally, CJ was supported by NSF grant DGE-203619. CK was supported by the NSF Graduate Research Fellowship under Grant No.\ 2140001. OL would like to thank Jack Hall and Fei Peng for many helpful discussions about perfect complexes and tor-dimension, and is furthermore extremely grateful to Fei Peng for the opportunity to participate in the 2023 Stacks Project Workshop.
SM would like to thank Alexander Efimov, Tony Pantev and Peter Scholze for helpful discussions and patience, and also Andres Fernandez Herrero for useful suggestions.
SS was partially supported by NSF grant DMS-2053473 through Kiran Kedlaya, and thanks Maxwell Johnson, Brendan Murphy and Nathan Wenger for discussions about various aspects of this paper.

\section{Defining the dual differential}
\label{sec: dual diff}
\subsection{Setup and motivation} 
Let $\theta\colon X\rightarrow S$ be a smooth morphism of schemes of relative dimension $n$. 
 Let $D\colon \mathcal{E}\rightarrow\mathcal{F}$ be a differential operator {of finite order} between two finite rank locally free $\O_X$-modules.  We define $\mathcal{E}^*\coloneqq \mathcal{E}^\vee\otimes \omega_{X/S}$, where $\omega_{X/S}\coloneqq \Omega^n_{X/S}$. 
 We wish to define a dual differential operator $D^*\colon\mathcal{F}^*\rightarrow \mathcal{E}^*$ such that $(D_1\circ D_2)^* = D_2^*\circ D_1^*$ and $(D_1^*)^*=D_1$.  Furthermore, for $D = d^0\colon\mathcal{O}_X \to \Omega^1_{X/S}$, we should have $D^* = -d^{n-1}\colon\Omega^{n-1}_{X/S} \to \omega_{X/S}$. 

 The existence of such a dual differential operator is motivated by the formal adjoint in the context of D-modules, which we now briefly recall (see e.g. \cite[Section 1.2]{Hotta}).  
 Given a smooth complex variety $X$ of dimension $n$, the sheaf $\Omega^n_{X}$ has a natural structure of a right D-module.
 This structure comes from the action of a vector field $\theta \in \Theta_X$ on $\omega\in \Omega^n_X$ given by $(\omega)\theta = -(\Lie \theta)\omega$.
 Then a differential operator between locally free sheaves $\theta\colon E\rightarrow F$ naturally gives a dual differential operator $F^\vee\otimes \Omega^n_X\rightarrow E^\vee\otimes \Omega^n_X$.  Locally, given an appropriate trivialisation and in rank one, if the differential operator is of the form $P(x, \partial) = \sum_{\alpha} a_{\alpha}(x)\partial^{\alpha}$, then its formal adjoint is given by 
 \[
\sum_{\alpha}(-\partial)^{\alpha}a_{\alpha}(x).
 \]
To generalise this definition to a differential operator relative to a morphism of schemes, de Jong indicated two methods in the blogpost \cite{blog}.  We will work out the second approach in detail here, which amounts to defining it locally and checking that it extends to a global definition.  This imitates the situation of D-modules which corresponds to the case $S=\Spec \mathbb{C}$; the main difference to be aware of is simply that in characteristic $p$, the differential operators can take a slightly different form.

\subsection{Local description}
\label{subsection: local description}
We begin in the local scenario where the base $S=\Spec A$ is affine.  Let $R$ be a smooth $A$-algebra with the structure morphism $\theta \colon \Spec R \to \Spec A$, and let $D \colon R^n\rightarrow R^m$ be an $A$-linear differential operator. 
By \cite[\href{https://stacks.math.columbia.edu/tag/039P}{Tag 039P}]{stacks-project}, locally $\theta$ can be factored as $\Spec R\xrightarrow{\pi}\Spec A[x_1, \ldots, x_n]\rightarrow \Spec A$ where $\pi$ is \'etale.
Explicitly, this factorisation is given by taking an open covering on which $\Omega^1_{X/S}$ is trivial and using $n$ linearly independent sections to give a map into $\mathbb{A}^n_S$.

By \cite[Theorem 16.11.2]{EGA4}, a differential operator $D$ is locally given by an $m\times n$ matrix with entries of the form 
 \[
\sum_{\alpha}f_\alpha \Big( \frac{1}{\alpha!}\cdot \frac{\partial}{\partial x_i^\alpha}\Big)_{\alpha}, 
 \]
 where $f_\alpha\in R$.
 Here $\alpha$ is a tuple of nonnegative integers, $x^\alpha := x_1 ^ {\alpha_1} \cdots x_n ^{\alpha_n}$, and $\alpha!$ denotes the product $\alpha_1! \cdots \alpha_n!$.
 For convenience we will first define the dual in the case that $\rank \mathcal{E}=\rank\mathcal{F} = 1$; locally this is equivalent to considering the case of differential operators $R\rightarrow R$. \\ 

Define $f^* = f$ and $(\frac{\partial}{\partial x_i})^* = -\frac{\partial}{\partial x_i}$, after making the appropriate trivialisation $\omega_{X/S}\cong \O_X$ defined by $g dx_1\wedge dx_2\wedge\cdots dx_n\mapsto g$ for $g\in R$.  Using this trivialisation, we extend this definition by the formula 

\begin{equation}
\label{eq: dual diff}
\left(\sum_{\alpha}f_\alpha \Big(\frac{1}{\alpha!}\cdot \frac{\partial}{\partial x^\alpha}\Big)_{\alpha}\right)^* 
= \sum_{\alpha} (-1)^{|\alpha|}\frac{1}{\alpha!}
\frac{\partial }{\partial x^\alpha}(f_\alpha \cdot \blank). 
\end{equation}

\begin{lemma}
    The dual of $d^0\colon\O_X\rightarrow \Omega^1_{X/S}$ is equal to $-d^{n-1}\colon \Omega^{n-1}_{X/S}\rightarrow \Omega^n_{X/S}$.
\end{lemma}

\begin{proof}
We do this locally, so assume $X\rightarrow S$ is of the form $\Spec R\xrightarrow{\pi}\Spec A[x_1, \ldots, x_n]\rightarrow \Spec A$ where $\pi$ is \'etale, as before.  
Indeed, locally we have $df = \sum_{i=1}^n \frac{\partial f}{\partial x_i}dx_i$, so identifying $\Omega^1_{X/S}$ with $Rdx_1\oplus \cdots \oplus Rdx_n$, we have
$d^0 = \left[ \frac{\partial }{\partial x_1}, \dots, \frac{\partial }{\partial x_n} \right]^T$.  

Then $(d^0)^*\colon (\Omega^{1}_{X/S})^\vee\otimes \Omega_{X/S}^n \rightarrow \O_X^\vee\otimes  \Omega^n_{X/S}$ is a map $\Omega_{X/S}^{n-1}\rightarrow \Omega_{X/S}^n$.
After making the natural identification 
\[(\Omega^{1}_{X/S})^\vee\otimes \Omega_{X/S}^n \cong R(dx_2 \wedge \cdots \wedge dx_n) \oplus  \cdots \oplus R((-1)^{n-1} dx_1\wedge \cdots\wedge dx_{n-1})\]
with alternating signs, we get that $(d^0)^*$ is defined by $\left[ -\frac{\partial }{\partial x_1}, \dots, -\frac{\partial }{\partial x_n} \right]$,
while $d^{n-1}$ is given by $\sum_{i=1}^n \frac{\partial}{\partial x_i}$ in this basis.
We thus have $(d^0)^* = - d^{n-1}$ as desired.
\end{proof}

\begin{proposition}
    Let $D_1, D_2 \colon R\rightarrow R$ be two finite order differential operators.  Then 
    
    (a) $(D_1\circ D_2)^* = D_2^*\circ D_1^*$.

    (b) $(D_1^*)^* = D_1$. 
\end{proposition}
\begin{proof}
    (a) 
    It suffices to show that
    $\left( f\frac{\partial}{\partial x_i^\alpha}g\frac{\partial}{\partial x_i^\beta} \right)^* 
    = (g\frac{\partial}{\partial x_i^\beta})^*(f\frac{\partial}{\partial x_i^\alpha})^*$.
    We prove it by induction on $|\alpha|$, with the base case $|\alpha|=0$.
    We note that when $D_1 = f$ and $D_2 = g \frac{\partial}{\partial x^\beta}$, then 
    \[ 
        \left( f\circ g\frac{\partial}{\partial x^\beta} \right)^*
        = \left( f g \circ \frac{\partial}{\partial x^\beta} \right)^*
        = (-1)^{|\beta|} \frac{\partial}{\partial x^\beta}(fg \cdot \blank)
        = \left(g \frac{\partial}{\partial x^\beta}\right)^* \circ f^* 
    . \]
    For the step of induction, given $\alpha$ and $\beta$ with $\alpha_i > 0$ for some $i$, we denote by $\alpha'$ the tuple $(\alpha_1, \dots , \alpha_i - 1 , \dots, \alpha_n)$, 
    and by $\beta'$ the tuple $(\beta_1, \dots , \beta_i + 1 , \dots, \beta_n)$.
    \begin{align*}
        \text{LHS}
        &= \left( f \dd{x^{\alpha'}} \dd{x_i} g \dd{x^\beta} \right)^*
        = \left( f \dd{x^{\alpha'}} 
        \left( \ddd{g}{x_i} + g \dd{x_i} \right)
        \dd{x^\beta} \right)^*
        \\ &= \left( f \dd{x^{\alpha'}} \circ \ddd{g}{x_i}
        \dd{x^\beta} \right)^*
        + \left( f \dd{x^{\alpha'}} \circ g \dd{x^{\beta'}} \right)^*
        \\ &=
        \left( \ddd{g}{x_i} \dd{x^\beta} \right)^* \circ
        \left( f \dd{x^{\alpha'}} \right)^*
        + \left( g \dd{x^{\beta'}} \right)^* \circ \left( f \dd{x^{\alpha'}} \right)^*
        \text{, by induction hypothesis}
        \\ &=
        \left( \ddd{g}{x_i} \dd{x^\beta} + g \dd{x_i} \dd{x^\beta} \right)^*
        \circ \left( f \dd{x^{\alpha'}} \right)^*
        \\ &=
        (-1)^{|\beta|}\dd{x^\beta} \circ \left( \ddd{g}{x_i}  - \dd{x_i} \circ g \right)
        \circ (-1)^{|\alpha'|} \dd{x^{\alpha'}}  \circ f
        \\ &=
        (-1)^{|\beta|}\dd{x^\beta} \circ g \circ \left( - \dd{x_i} \right)
        \circ (-1)^{|\alpha'|} \dd{x^{\alpha'}}  \circ f
        \\ &=
        (-1)^{|\beta|}\dd{x^\beta} \circ g
        \circ (-1)^{|\alpha|} \dd{x^{\alpha}} \circ f
        \\ &= \text{RHS, by definition.} 
    \end{align*}

    (b) Applying part (a) to the term corresponding to each fixed $\alpha$, we get
    \begin{align*}
    \left( f \dd{x^\alpha} \right)^{**}
    &= (-1)^{|\alpha|} \left( \dd{x^\alpha} \circ f \right)^{*}
    \\ &= (-1)^{|\alpha|} f^* \circ \left( \dd{x^\alpha} \right)^*
    \text{, by part (a)}
    \\ &= f \dd{x^\alpha} ,
    \end{align*}
    and the result follows.  
\end{proof}

The same method allows us to define the dual differential locally for locally free sheaves of finite rank by doing so coordinate-wise.  The next goal is to show that these local definitions do indeed globalise.
  
\subsection{Global definition}
To give a global definition of the dual differential in which $S$ is an arbitrary scheme, we must ensure that the local definitions described in the previous subsection patch together to uniquely define a dual differential.  
Given a smooth morphism $f\colon X\rightarrow S$ (where $S$ is no longer assumed to be affine), let $\gamma\colon \mc{E}\rightarrow \mc{F}$ a differential operator between locally free sheaves of finite rank on $X$.  
Let us first consider how the dual differential is locally defined in the case $\mc{E}=\mc{F} = \O_X$.  
First, we restrict to an affine open covering $\{U_{\alpha}\}$ of $X$ with $\Omega^1_{X/S}$ locally free over $\{U_{\alpha}\}$ and the image of each $U_{\alpha}$ contained in an open affine subscheme of $S$.  Then we choose a trivialisation that gives a factorisation $U_\alpha\xrightarrow{\pi_{\alpha}} \mb{A}^n_S\rightarrow S$ with $\pi_{\alpha}$ \'etale, and finally we apply the formula given by Equation \ref{eq: dual diff} on each $U_{\alpha}$. 
If we show that these are compatible, then in fact this leads to a definition of the dual differential $\gamma^*$ for any $\mc{E}$ and $\mc{F}$ locally free of finite rank.  Indeed, we can further refine the open cover to one where $\mc{E}$ and $\mc{F}$ are locally free.  Then the compatibilities on each matrix component of the dual differential on the refined cover yield a well-defined definition of $\gamma^*$. 

\begin{proposition}
    Applying the rule described in Equation \ref{eq: dual diff} to a differential operator $D\colon \O_X\rightarrow \O_X$ gives a compatible system of dual differentials between different charts, and thus this method extends to give a well-defined dual differential for locally free sheaves of finite rank.
\end{proposition}
\begin{proof}
As described above, the second statement follows from the first.
To show compatibility, it suffices to check that for any two $U_{\alpha}, U_{\beta}$, the restrictions of $D^*|_{U_{\alpha}}$ and $D^*|_{U_\beta}$ to $U_{\alpha}\cap U_{\beta}$ coincide.
For this, it suffices to check this on any open cover of $U_{\alpha}\cap U_{\beta}$ itself.
We may choose one consisting of affine $U = \Spec R\subset U_{\alpha}\cap U_{\beta}$ where $U$ maps into some open affine $S' = \Spec A$.
Then the two factorisations $U\xrightarrow{}\mb{A}^n_S \rightarrow S$ are determined by two distinct trivialisations of $\Omega^1_{U/S}$, which differ by an element of $\GL(n, A)$.  

Let $dx_1\wedge \cdots dx_n$ and $dy_1\wedge \cdots dy_n$ be the two trivialisations, with $J = \det [\frac{\partial x_i}{\partial y_j}]$ being the Jacobian.  Let us first check the result for $D=\frac{\partial}{\partial x_1}$.  We have 
\[D_x^*(f dx_1\wedge\cdots \wedge dx_n) = -\frac{\partial f}{\partial x_1} dx_1\wedge\cdots \wedge dx_n \]
Since $D_y = \sum_{k=1}^n\frac{\partial y_k}{\partial x_1}\frac{\partial}{\partial y_k}$, we have $D_y^* = \sum_k -\frac{\partial}{\partial y_k}\frac{\partial y_k}{\partial x_1}$.  Thus
\begin{align*}
   D_y^*(f dx_1\wedge\cdots \wedge dx_n) &= D_y^*(f J dy_1\wedge \cdots \wedge dy_n) \\
   &= - \sum_k \frac{\partial}{\partial y_k} \left(\frac{\partial y_k}{\partial x_1} f J\right) dy_1\wedge \cdots \wedge dy_n \\ 
   &= - \left(\sum_k \frac{\partial f}{\partial y_k} \left(\frac{\partial y_k}{\partial x_1} J\right) + f \frac{\partial}{\partial y_k} \left(\frac{\partial y_k}{\partial x_1} J\right) \right) dy_1\wedge \cdots \wedge dy_n \\
   &= -\frac{\partial f}{\partial x_1} dx_1\wedge\cdots \wedge dx_n - f\sum_k \frac{\partial}{\partial y_k} \left(\frac{\partial y_k}{\partial x_1} J\right) dy_1\wedge \cdots \wedge dy_n \\ 
    &= -\frac{\partial f}{\partial x_1} dx_1\wedge\cdots \wedge dx_n 
 \end{align*}
as desired, where the last equality holds because $\frac{\partial y_k}{\partial x_1} J$ is an element of $A$ 

Furthermore, it is clear that the result holds when $D$ is an element of $R$.  Finally, since $(\cdot)^*_x$ and $(\cdot)^*_y$ are both anti-involutions, if the result holds for two differential operators, then it holds for their composition.  Since elements of $R$ and differential operators of the form $\frac{\partial}{\partial x_i}$ generate all differential operators, we are done.  
\end{proof}

\subsection{Coordinate-free definition of the dual differential}

We now briefly explain how our definition of dual in coordinates arises from the theory of duality outlined in \cite{blog}.
We first recall the EGA definition of a differential operator, for which we fix the following notation: $\theta \colon X \to S$ is morphism of schemes, $X^{(n)}$ denotes the $n$th infinitesimal neighborhood of the diagonal of $X \times_S X$, and we have the following commutative diagram.
\[\begin{tikzcd}
	& {X^{(n)}} \\
	& {X \times_S X} \\
	X && X
	\arrow["{h_n}"', from=1-2, to=2-2]
	\arrow["{p_n}"', curve={height=12pt}, from=1-2, to=3-1]
	\arrow["{q_n}", curve={height=-12pt}, from=1-2, to=3-3]
	\arrow["p"', from=2-2, to=3-1]
	\arrow["q", from=2-2, to=3-3]
\end{tikzcd}\]

The underlying topological space of $X^{(n)}$ is the same as that of $X$, so denoting the structure sheaf of $X^{(n)}$ by $\cP^n$, we can write $X^{(n)} = (X , \cP^n)$.
Then the projections $p_n$ and $q_n$ induce the morphisms $p_n^\#$ and $q_n^\# \colon \cO_X \to \cP^n$, which
endow $\cP^n$ with two structures of an $\cO_X$-module.

\begin{definition}[{\cite[Def. 16.8.1]{EGA4}}]
    Given two $\cO_X$-modules $\cE$ and $\cF$, a \emph{differential operator} is an $\cO_S$-linear map $D \colon \cE \to \cF$ that can be factored as $D = u \circ (q_n^\# \otimes 1_\cE)$ for some $\cO_X$-linear map $u \colon \cP^n \otimes_{q_n, \cO_X} \cE \to \cF$.
    Here the $\cO_X$-linearity is considered with respect to the module structure on $\cP^n$ given by $p_n^\#$, while the tensor product is taken with respect to one given by $q_n^\#$.
\end{definition}

\begin{construction}
\label{construction: coordinate-free dual}
    Given a smooth morphism $\theta \colon X \to S$ and a differential operator $D \colon \cE \to \cF$ corresponding to $u \colon \cP^n \otimes_{q_n, \cO_X} \cE \to \cF$, we have a canonical identification $\cF \cong \cF^{**}$, hence by tensor-Hom adjunction, $u$ corresponds to the morphism $\cF^* \otimes_{\cO_X, p_n} \cP^n \otimes_{q_n, \cO_X} \cE \to \omega_{X/S}$.
    By another tensor-Hom adjunction, the latter corresponds to the morphism denoted 
    \[
        u^* \colon \cF^* \otimes_{\cO_X, p_n} \cP^n \to \cE^*
    , \]
    which is $\cO_X$-linear with respect to the module structure on $\cP^n$ given by $q_n^\#$.
\end{construction}

\begin{remark}
    One can check that in the setting of \cref{construction: coordinate-free dual}, the right adjoint to $Rp_{n*}$ is given by $p_n^! = \localRHom_{p_n, \cO_X} ( \cP^n , \blank )$, and similarly for $q_n$.
    It then follows that \cref{construction: coordinate-free dual} gives rise to the same notion of dual differential operators as described in \cite{blog}.
\end{remark}

\begin{lemma}
    Let $S = \Spec A$ and $X = \Spec R$ for $R = A[x_1, \dots, x_n]$, and fix a differential operator $D \colon R \to R$ corresponding to some $u \colon \cP^n \otimes_{q_n, R} R \to R$.
    Then $D^*$ as defined in \cref{subsection: local description} is the differential operator corresponding to the morphism $u^*$ as defined in \cref{construction: coordinate-free dual},
    i.e. $D^* = u^* \circ (1_\cF \otimes p_n^\#)$.
\end{lemma}

\begin{proof}
    It's enough to verify the claim for $D = \dd{x_1}$.
    Then $n=1$, and if $I = (1 \otimes x_i - x_i \otimes 1)_{i}$ denotes the ideal of the diagonal, then
    \[ \cP^1 = (R \otimes_A R) / I^2 \cong R \cdot 1 \oplus \bigoplus_{i=1}^{n} R \cdot (1 \otimes x_i - x_i \otimes 1).
    \]
    Then one can check that the corresponding $u \colon \cP^1 \otimes_{q_n, R} R \cong \cP^1 \to R$ is defined as projection on the component $R \cdot (1 \otimes x_1 - x_1 \otimes 1)$, e.g.
    \[
        u \circ (q_n^\# \otimes 1_\cE) (x_1) = u(1 \otimes x_1) = u( x_1 \cdot 1 + 1 \cdot (1 \otimes x_1 - x_1 \otimes 1)) = 1.
    \]
    Under identification $\cP^1 \otimes_{q_n, R} R \cong \cP^1 \cong R \otimes_{p_n, R} \cP^1$, the morphism $u^*$ is identified with $u$.
    Finally, a direct calculation shows that $\left(\dd{x_1} \right)^* = - \dd{x_1}$, e.g.
    \[
        u^* \circ (1_\cF \otimes p_n^\#) (x_1) = u^* (x_1 \otimes 1) = u^* ( x_1 \cdot 1 + (-1) \cdot (1 \otimes x_1 - x_1 \otimes 1)) = -1 .
        \qedhere
    \]
\end{proof}

\section{Filtered K\"unneth formula}
\label{sec: kunneth}

%
\subsection{Setup}
\label{subsection: setup of filtered Kuenneth}
Given a morphism of schemes $\theta \colon X\rightarrow S$, we want to study objects that generalise the de Rham complex, namely complexes of quasi-coherent sheaves on $X$ whose differentials are $S$-differential operators.
Moreover, we would like to consider filtrations on such objects.
The theory of filtered derived categories
we present now is based on the treatment in Illusie's thesis \cite[Chapter V]{Illusie} and \cite{stacks-project}*{05RX}.
These sources mainly treat the case of the filtered derived category of an abelian category, while the category of coherent $\cO_X$-modules with differential operators as morphisms is not abelian.
In order to sidestep this issue, we first define, for a pair of abelian categories $\cA$ and $\cS$ with an additive functor $\cA \to \cS$, the homotopy category of what we call \emph{$\cS$-complexes}, namely $\ZZ$-indexed tuples of objects in $\cA$ whose images in $\cS$ form a complex.
Since we would like to require $\cS$ to be abelian, we apply this construction to $\cA = \QCoh X$ and $\cS = \Mod (\theta^{-1} \cO_S)$, then take the full subcategory of those $\cS$-complexes whose differentials are $S$-differential operators.
The homological algebra that ensues will be done in this category and its localisation with respect to filtered quasi-isomorphisms.

To set the stage, we consider the following cartesian square 
\begin{equation}
\label{commdiag: setup}
\begin{tikzcd}
	& {X \times_S Y} \\
	X && Y \\
	& S
	\arrow["p"', from=1-2, to=2-1]
	\arrow["q", from=1-2, to=2-3]
	\arrow["\theta"', from=2-1, to=3-2]
	\arrow["\eta", from=2-3, to=3-2]
	\arrow["f"{description}, from=1-2, to=3-2]
\end{tikzcd}
\end{equation}

and assume that $X$ and $Y$ are quasi-compact with affine diagonal and
$S = \Spec A$ is affine. 
We further assume that $L$ (resp. $M$) is a bounded $S$-differential complex
in the filtered derived category, consisting of $S$-flat $\O_X$- (resp. $\O_Y$-) modules (we will define these terms precisely).    

The goal of this section is, first, to make sense of the formula
\[
R\theta_*(L) \Lotimes_{\O_S} R\eta_*(M) \to Rf_* (\Tot (L \boxtimes M)),
\]
and second, to prove that it is an isomorphism under certain assumptions.

In the sequel, we will apply this to the case $X=Y$ and
$L = M = (\cE \xrightarrow{D} \cF)$.  

\subsection{Filtered derived categories}

\begin{definition}
    For an abelian category $\A$, we will define the \emph{filtered derived category $\DF(\cA)$ with uniformly bounded filtrations} following the same process as in \cite{stacks-project}*{\href{https://stacks.math.columbia.edu/tag/05RX}{05RX}}.
    \begin{itemize}
        \item Let $\operatorname{C}(\A)$ be the category of complexes of objects in $\cA$.
        \item $\Fil^f(\A)$ is the category of finite filtered objects in $\cA$, i.e. objects with finite filtrations (see \cite{stacks-project}*{\href{https://stacks.math.columbia.edu/tag/05RY}{05RY}}).
        \item Let $\CF(\cA)$ be the full subcategory of $\operatorname{C}(\Fil^f(\A))$ consisting of complexes $L$ for which
        the filtration is uniformly bounded, i.e. there exists $N > 0$ such that $F^{-N} L = L$ and $F^N(L) = 0$.
        \item 
        We denote by $\KF(\cA)$ the homotopy category of $\CF(\cA)$.
        \item Subcategory $\FAc(\cA)$ and multiplicative system $\FQis(\cA)$ are defined analogously to \cite{stacks-project}*{\href{https://stacks.math.columbia.edu/tag/05RZ}{05RZ}, \href{https://stacks.math.columbia.edu/tag/05S1}{05S1}}.
        In particular, a morphism $\varphi$ is a \textit{filtered quasi-isomorphism} if $\operatorname{gr}(\varphi)$ is a quasi-isomorphism, and we write 
        $\varphi \in \FQis(\cA)$.
        \item We then define 
        $\DF(\cA) = \KF(\cA) / \FAc(\cA) =
        \left( \FQis(\cA) \right)^{-1} \KF(\cA)$.
    \end{itemize}
\end{definition}

\begin{remark}
    An adaptation of the standard argument (e.g., \cite{stacks-project}*{\href{https://stacks.math.columbia.edu/tag/05RW}{05RW}}) shows that $\DF(\cA)$ is naturally identified with a full subcategory of $\operatorname{DF}(\cA)$ (as defined in \cite[{\href{https://stacks.math.columbia.edu/tag/05S2}{05S2}}]{stacks-project}).
\end{remark}

\subsection{Filtered \texorpdfstring{$\cS$}{S}-complexes}


We will now formalise the setting described in \cref{subsection: setup of filtered Kuenneth}.
Let $\cA$ and $\cS$ be two abelian categories, and let $\sigma \colon \cA \to \cS$ be an additive functor.
We want to make sense of what it means to be a complex of objects in $\cA$ whose differentials are in $\cS$.
We recall that 
$\Fil^f(\cA)$ is the category of \emph{finite filtered objects} in $\cA$, i.e. objects with finite filtrations (see \cite{stacks-project}*{\href{https://stacks.math.columbia.edu/tag/05RY}{05RY}}).

\begin{definition}
    We define $\CF(\cA, \cS)$ to be the category whose objects are tuples 
    \[
        (L^i, d^i \colon \sigma (L^i) \to \sigma (L^{i+1}))_{i \in \ZZ}
        \text{, where}
    \]
    \begin{itemize}
        \item $L^i \in \Fil^f (\cA)$,
        \item $(\sigma(L) , d)$ is a complex of objects in $\Fil^f (\cS)$, and
        \item the filtration of $L$ is uniformly bounded.
    \end{itemize}
    A morphism $\varphi \colon (L,d_L) \to (M, d_M)$ is given by a tuple $\varphi \colon L^i \to M^i$ of morphisms in $\Fil^f (\cA)$ such that $\sigma(\varphi)$ defines a morphism of complexes over $\Fil^f (\cS)$.
    We call $\CF (\cA , \cS)$ the \emph{category of filtered $\cS$-complexes over $\cA$ with uniformly bounded filtrations}.
\end{definition}

\begin{remark}
    One can similarly define the \emph{category of $\cS$-complexes over $\cA$} without filtrations, and the subsequent lemmas will still hold for formal reasons.
    Since we won't use it in the sequel, we omit the obvious definition and statements for the sake of brevity.
\end{remark}

\begin{definition}
    Given two $\cS$-complexes $L, M \in \CF(\cA, \cS)$ and a morphism $\varphi \colon L \to M$, we say that $\varphi$ is \emph{null-homotopic} if there exists a collection of morphisms 
    $h = \left(h^i \colon L^i \to M^{i-1} \right)$ in $\cA$ such that $d_M \circ \sigma (h) + \sigma (h) \circ d_L = \sigma (\varphi)$ in $\cS$.
    We denote $\KF (\cA, \cS)$ the quotient of $\CF (\cA, \cS)$ by null-homotopic maps, and call it the \emph{homotopy category of filtered $\cS$-complexes over $\cA$ with uniformly bounded filtrations}.
\end{definition}

\begin{definition}[{cf. \cite{stacks-project}*{\href{https://stacks.math.columbia.edu/tag/05RZ}{05RZ}}}]
    An $\cS$-complex $L \in \KF (\cA , \cS)$ is called \emph{filtered acyclic} if the complex $\gr(\sigma(L))$ is acyclic.
    The full subcategory of such will be denoted $\FAc (\cA , \cS)$.
    A morphism of $\cS$-complexes $\varphi \colon L \to M$ is called a \emph{filtered quasi-isomorphism} if $\gr(\sigma(\varphi))$ is a quasi-isomorphism.
    The collection of all such morphisms will be denoted $\FQis(\cA, \cS)$.
\end{definition}

Similarly to \cite{stacks-project}*{\href{https://stacks.math.columbia.edu/tag/014E}{014E}}, we define the \emph{cone} $C(\varphi)$ of a morphism $\varphi \colon L \to M$ as a pair $(M \oplus L[1], d_{C(\varphi)})$, where
$d_{C(\varphi)} = \left[ \begin{smallmatrix}
d_M & \sigma (\varphi) \\ 0 & -d_L
\end{smallmatrix} \right]$.

\begin{lemma}
\label{lemma: KF is triangulated}
    Let $\cA$ and $\cS$ be two abelian categories, and let $\sigma \colon \cA \to \cS$ be an additive functor.
    Then $\KF(\cA, \cS)$ with the collection of distinguished triangles 
    \[
        \left\{ L \xrightarrow{\varphi} M 
        \xrightarrow{\left[ \begin{smallmatrix} 1 \\ 0 \end{smallmatrix} \right]}
        C(\varphi)
        \xrightarrow{\left[ \begin{smallmatrix} 0 & 1 \end{smallmatrix} \right]} L[1] \right\},
    \]
    where $\varphi$ ranges over all morphisms of $\cA$,
    forms a triangulated category.
\end{lemma}

\begin{proof}
    One observes that the standard construction of necessary homotopies (e.g. \cite{stacks-project}*{\href{https://stacks.math.columbia.edu/tag/014S}{014S}} and references therein) when checking the axioms (TR1--TR4) still works,
    i.e. they do not use components of differentials and hence exist on the level of $\cA$.
\end{proof}

\begin{remark}
\label{remark: FAc is a saturated subcat}
    Analogously to the proof of \cite{stacks-project}*{\href{https://stacks.math.columbia.edu/tag/05S1}{05S1}}, one observes that $H^0 \circ \gr \circ \sigma$ is a homological functor $\KF (\cA, \cS) \to \cS$,
    hence by \cite{stacks-project}*{\href{https://stacks.math.columbia.edu/tag/05RM}{05RM}},
    $\FAc(\cA, \cS)$ is a saturated triangulated subcategory of $\KF(\cA, \cS)$, and the corresponding saturated multiplicative system is $\FQis (\cA , \cS)$.
    One can then define the filtered derived category of $\cS$-complexes; we abstain from doing it, for we need a further refinement to suit our purposes. 
\end{remark}


\subsection{\texorpdfstring{$S$}{S}-differential complexes.}

Given an $S$-scheme $X \xrightarrow{\theta} S$,
we can consider the category $\CF (\QCoh (X), \Mod(\theta^{-1}\cO_S))$ of filtered $\Mod(\theta^{-1}\cO_S)$-complexes with respect to the natural forgetful functor
$\sigma \colon \QCoh (X) \to \Mod(\theta^{-1}\cO_S)$.

\begin{definition}
    For an $S$-scheme $X \xrightarrow{\theta} S$,
    we define a \emph{filtered $S$-differential complex} $L$ 
    as an object of $\CF (\QCoh (X), \Mod(\theta^{-1}\cO _S))$ such that in addition the differentials 
    \[ d^i \colon \sigma(L^i) \to \sigma(L^{i+1}) \]
    come from finite order $S$-differential operators between the $\cO_X$-modules $D^i \colon L^i \to L^{i+1}$, i.e $d^i = \sigma (D^i)$.
    We denote the full subcategory of $S$-differential complexes by $\CFD (X,S)$.
    Using the notion of null-homotopy from $\CF (\QCoh (X), \Mod(\theta^{-1}\cO _S))$, we define the \emph{homotopy category of filtered $S$-differential complexes} $\KFD (X,S)$ as the quotient of $\CFD (X,S)$ by null-homotopic maps.
\end{definition}

Since the forgetful functor $\sigma$ doesn't change the underlying sheaf of sections, we will suppress it from the notation from now on.

The results of \cref{lemma: KF is triangulated} and \cref{remark: FAc is a saturated subcat} hold for $\KFD (X,S)$ as well,
and we denote
\begin{align*}
    \FDAc (X,S) &:= \FAc (\QCoh (X), \Mod(\theta^{-1}\cO_S)) \cap \KFD (X,S) , \\
    \FDQis (X,S) &:= \FDQis (\QCoh (X), \Mod(\theta^{-1}\cO_S)) \cap \KFD (X,S) .
\end{align*}
With this, we are in a position to make the following definition.

\begin{definition}
    The \emph{filtered derived category of $S$-differential complexes with uniformly bounded filtrations} is 
    \[
    \DFD(X,S) := 
    \KFD(X,S) / {\FDAc}(X,S) =
        \left( \FDQis (X,S) \right)^{-1} \KFD (X,S).
    \]
\end{definition}



Since pulling along $p \colon Z \to X$ is an $X$-linear operation, it is not automatic that pulling back an $S$-linear differential operator will be well-defined.
However, the pullback along a projection map $p\colon X\times_S Y\rightarrow X$ will take an $S$-differential complex on $X$ to one on $X\times_S Y$.  In fact, we have the following lemma.   

\begin{lemma}
    With notation as in \cref{commdiag: setup}, 
    given filtered $S$-differential complexes $L$ and $M$ on $X$ and $Y$, respectively, we can define the filtered $S$-differential double complex $L \boxtimes M$ with
    \[
        (L \boxtimes M)^{i,j} = 
        p^* (L^i) \otimes_{\mathcal{O}_{X \times_S Y}}
        q^* (M^j)
    \]
    and differentials 
    \[
        d_L \boxtimes 1_M
        \quad \text{and} \quad
        1_L \boxtimes d_M.
    \]
\end{lemma}

\begin{proof}
    Apply \cite{stacks-project}*{\href{https://stacks.math.columbia.edu/tag/0G45}{0G45}}.
    \qedhere
\end{proof}

We note that the construction in this lemma is carried out for complexes in $\CFD (X,S)$ and $\KFD (X,S)$, and we do not attempt to pass it to the level of derived categories.

\subsection{Derived functors for filtered derived categories} 
To prove existence of the derived tensor product for a filtered derived category, and to compute it,
we show that for each object, we can find a filtered quasi-isomorphism from a complex which is filtered K-flat (i.e. adapted to computing derived tensor product).

\begin{lemma}
\label{lemma: filtered resolution with surjections}
    Let $\mathcal{A}$ be a cocomplete abelian category with exact direct sums, and let $\cA \to \cS$ be an additive functor that commutes with arbitrary direct sums. 
    Fix a subclass of objects $\mathcal P$ in $\cA$ closed under arbitrary direct sums.
    Assume that every object of $\cA$ functorially admits a surjection from an object of $\mathcal P$.
    Then any filtered complex $L \in \CF(\mathcal A)$ or $L \in \CF(\cA,\cS)$ admits a filtered quasi-isomorphism $P(L) \to L$ such that:
    \begin{itemize}
        \item $P(L) \to L$ is term-wise surjective, and
        \item for every $i, m \in \mathbb Z$, both $F^i(P(L)^m)$ and $\operatorname{gr^i}(P(L)^m)$
        are in $\mathcal P$;
        \item moreover, if the objects in $\cP$ were adapted to computing a certain derived functor commuting with arbitrary direct sums, then $P(L)$, $F^i(P(L))$ and $\operatorname{gr^i}(P(L))$ are also adapted to computing this derived functor.
    \end{itemize}
\end{lemma}

\begin{proof}
    By a standard argument (e.g., the dual to \cite{stacks-project}*{\href{https://stacks.math.columbia.edu/tag/013K}{013K} and \href{https://stacks.math.columbia.edu/tag/079P}{079P}}), existence of functorial surjections from $\mathcal P$ implies that any complex $L \in C(\A)$ admits a functorial term-wise surjection from a complex of objects from $\mathcal P$, which we denote by $\pi \colon P'(L) \to L$, such that $P'(L)$ is adapted to computing the left derived functor. Further, for each step of filtration, we get its own resolution, and thus, by functoriality of $P'$, we have a commutative diagram of complexes
\[\begin{tikzcd}
	{P'(F^i(L))} && {F^i(L)} \\
	\\
	{P'(L)} && L
	\arrow["{\pi^i}", two heads, from=1-1, to=1-3]
	\arrow["{\pi}", two heads, from=3-1, to=3-3]
	\arrow["{P(\iota^i)}"', from=1-1, to=3-1]
	\arrow["\iota^i", hook, from=1-3, to=3-3]
\end{tikzcd}\]
    where the horizontal arrows are term-wise surjections. Note that 
    $P'(\iota^i) \colon P'(F^i(L)) \to P'(L)$ is not necessarily an injection, so we will modify $P'(L)$ into $P(L)$ in such a way that the $P'(F^i(L))$ induce a desired filtration on $P(L)$. We will adapt the use of the Rees construction as was done in \cite{MR3581172}*{Section 3}.

    To this end, define 
    \[
        \operatorname{Rees}(i) := \bigoplus_{j \geq i} P(F^j(L))
        .
    \]
    Note that there is a natural identification $\operatorname{Rees}(i) \cong P'(F^i(L)) \oplus \operatorname{Rees}({i+1})$. We use it to define the morphism $\alpha \colon \operatorname{Rees}({i+1}) \to \operatorname{Rees}(i)$ as 
    \[
    \alpha :=
    \begin{bmatrix}
        P'(\iota) \\ 1
    \end{bmatrix}
    \oplus 1_{\operatorname{Rees}(i+2)},
    \]
    where
    \[
    \begin{bmatrix}
        P'(\iota) \\ 1
    \end{bmatrix}
    \colon 
    P'(F^{i+1}(L)) \to
    P'(F^{i}(L)) \oplus P'(F^{i+1}(L))
    ,
    \]
    and $\iota$ denotes the inclusion $\iota = \iota^{i+1,1} \colon F^{i+1}(L) \to F^i(L)$.
    Denote by $Q(i)$ the cone of this morphism:
    \[
        \operatorname{Rees}({i+1}) \xrightarrow{\alpha}
        \operatorname{Rees}({i}) \to Q(i) .
    \]
    We note that $Q(i)$ is quasi-isomorphic to $P'(F^i(L))$, because the above triangle is quasi-isomorphic to the direct sum of two distinguished triangles: $P'(F^{i+1}) \to P'(F^{i}) \oplus P'(F^{i+1}) \to P'(F^{i})$ and $\operatorname{Rees}(i+2) \to \operatorname{Rees}(i+2) \to 0$.

    Now we can define the desired complex $P(L)$. Since we are interested in the case of uniformly bounded filtrations, we can choose $N$ such that $F^{-N}(L) = L$ and $F^N(L) = 0$. We then set
    \[
    F^i(P(L)) = 
    \begin{cases}
        Q(-N) & \text{for } i < -N, \\
        Q(i) & \text{for } -N \leq i \leq N,\\
        0 & \text{for } N < i.
    \end{cases}
    \]
    This constructions yields a filtered complex $P(L) : = Q(-N)$, where the inclusions $Q(i+1) \to Q(i)$ are given by inclusions of direct summands. Each $F^i(P(L))$ is quasi-isomorphic to $P'(F^i(L))$ and hence to $F^i(L)$.

\end{proof}

\begin{lemma}
\label{lemma: deriving filtered tensor}
    Let $(S, \mathcal O_S)$ be a ringed space.
    Then there is a derived functor $\otimes^{\mb{L}}_{\mathcal O_S}$ defined on the filtered derived category $DF_b (\mathcal O_S) := DF_b (\operatorname{Mod} (\mathcal O_S))$ with uniformly bounded filtrations, and this functor commutes with $\operatorname{gr}$:
    \[\begin{tikzcd}
    	{DF_b (\mathcal O_S) \times DF_b (\mathcal O_S)} && {DF_b (\mathcal O_S)} \\
    	\\
    	{\displaystyle \left( \bigoplus_{\mathbb Z} D (\mathcal O_S) \right) \times \left( \bigoplus_{\mathbb Z} D (\mathcal O_S) \right)} && {\displaystyle \left( \bigoplus_{\mathbb Z} D (\mathcal O_S) \right)}
    	\arrow["{\otimes^{\mb{L}}_{\mathcal O_S}}", from=1-1, to=1-3]
    	\arrow["{\operatorname{gr}}"', from=1-1, to=3-1]
    	\arrow["{\otimes^{\mb{L}}_{\mathcal O_S}}", from=3-1, to=3-3]
    	\arrow["{\operatorname{gr}}", from=1-3, to=3-3]
    \end{tikzcd}\]
\end{lemma}

\begin{proof}
    We prove both claims by construction. 
    By 
    \cref{lemma: filtered resolution with surjections},
    every complex $L$ with uniformly bounded filtrations admits a filtered surjective quasi-isomorphism $f\colon P(L) \rightarrow L$ from a K-flat filtered complex $P$ such that  $\operatorname{gr}(P)$ is K-flat as well.
    This shows that $\otimes^{\mb{L}}_{\mathcal O_S}$ exists and can be computed by
    \[
        L \otimes^{\mb{L}}_{\mathcal O_S} L' = P(L) \otimes_{\mathcal O_S} L',
    \]
    and since $\operatorname{gr}(P(L))$ is K-flat, we also have
    \[
        \operatorname{gr}(L) \otimes^{\mb{L}}_{\mathcal O_S} \operatorname{gr} (L') 
        = \operatorname{gr}(P(L)) \otimes_{\mathcal O_S} \operatorname{gr} (L')
        = \operatorname{gr}(P(L) \otimes_{\mathcal O_S} L')
        = \operatorname{gr}(L \otimes^{\mb{L}}_{\mathcal O_S} L').
        \qedhere
    \]
  
\end{proof}

\begin{corollary}
    Let $\theta \colon X \to S$ be a morphism of ringed spaces. Then there exists a derived functor
    \[
        L\theta^* \colon DF_b (\mathcal O_S) \to DF_b (\mathcal O_X),
    \]
    and it commutes with $\operatorname{gr}$.
\end{corollary}

\begin{proof}
    The functor $\theta^{-1} \colon \operatorname{Mod}(\mathcal O_S) \to \operatorname{Mod}(\mathcal O_X)$ is exact and takes $\mathcal O_S$-flat objects to $\theta^{-1}\mathcal O_S$-flat objects. 
    Now we apply \cref{lemma: deriving filtered tensor} to $\theta^{-1}( \blank ) \otimes_{\theta^{-1} \mathcal O_S} \mathcal O_X$, which is the definition of $ \theta^*(\blank)$.
\end{proof}

\begin{lemma}
\label{lemma: deriving filtered pushforward}
    Let $f \colon Z \to S$ be a morphism of ringed spaces. Then there exists a derived functor
    \[
        Rf_* \colon DF_b (\mathcal O_Z) \to DF_b (\mathcal O_S),
    \]
    and it commutes with $\operatorname{gr}$.
\end{lemma}

\begin{proof}
    The dual to \cref{lemma: filtered resolution with surjections} holds by a dual argument, or alternatively, by a combination of \cite{stacks-project}*{\href{https://stacks.math.columbia.edu/tag/05TW}{05TW},
    \href{https://stacks.math.columbia.edu/tag/079P}{079P},
    \href{https://stacks.math.columbia.edu/tag/070L}{070L}} in the case when the chosen subclass of objects consists of injectives; in the latter argument, some care is needed to control that the resulting filtered K-injective complex has uniformly bounded filtrations.
    Therefore, the argument dual to one in the proof of \cref{lemma: deriving filtered tensor} applies.
\end{proof}

\subsection{The cup product morphism}

\begin{proposition}
    Consider $S$-schemes $X$ and $Y$ as in \cref{commdiag: setup}. We assume that $X$ and $Y$ are quasicompact with affine diagonal, and $L$ (resp. $M$) is a filtered $S$-differential complex on $X$ (resp. $Y$).
    Then there exists a cup product morphism
    \[
        R\theta_* (L) \otimes^{\mb{L}}_{\mathcal O_S} R\eta_* (M) 
        \to
        Rf_* \left( \Tot (L \boxtimes M) \right).
    \]
    If in addition $L$ and $M$ consist of flat objects, have uniformly bounded filtration, $L$ is bounded and $M$ is bounded below, then 
    this morphism induces a quasi-isomorphism on the associated graded objects:
    \[
        \operatorname{gr} \left(  R\theta_* (L) \otimes^{\mb{L}}_{\mathcal O_S} R\eta_* (M) \right)
        \to
        \operatorname{gr} \left( Rf_* \left( \Tot (L \boxtimes M) \right) \right).
    \]
\end{proposition}

\begin{proof}
    We have verified that all the operations in the construction of \cite{stacks-project}*{\href{https://stacks.math.columbia.edu/tag/0G4B}{0G4B}} are well-defined on filtered derived categories.
    Note that apply the construction of \cref{lemma: deriving filtered pushforward} to morphisms $(X, \theta^{-1} \cO_S) \to S$, $(Y, \eta^{-1} \cO_S) \to S$, $(X \times_S Y, f^{-1} \cO_S) \to S$.
    
    The second claim follows from \cite{stacks-project}*{\href{https://stacks.math.columbia.edu/tag/0FLT}{0FLT}} in case $S$ is affine, since we checked that taking derived functors commutes with $\operatorname{gr}$.
    If $S$ is not affine, we check the quasi-isomorphism locally over open affine subschemes $S' \subset S$.
\end{proof}

\section{Construction of the copairing map for locally free sheaves} 
\label{sec: copairing locally free}
Let $X$ over $S = \Spec A$ be smooth proper of relative dimension $n$.
Let $\gw_{X/S} = \Omega^n_{X/S}$, let $\mc{E}$ be a locally free sheaf on $X$, and let $\mc{E}^* := \mathscr{H}om_{\mc{O}_X}(\mc{E}, \gw_{X/S})$.  We have the Serre duality pairing
\be
\ep \colon R \Gamma(X, \mc{E}^*)[n] \tten_A R \Gamma(X, \mc{E}) \to A,
\ee
which is a perfect pairing of perfect complexes in $D(A)$. The goal of this section is to explicitly construct a map 
\be
\eta \colon A \to R \Gamma(X, \mc{E}^*)[n] \tten_A R\Gamma(X, \mc{E}),
\ee
which we will show to be a copairing map in the following section.  
First, we note that we have the K\"unneth isomorphism
\be
R\Gamma(X, \mc{E}^*)[n] \tten_A R \Gamma(X, \mc{E}) \cong R \Gamma(X \times_S X, \mc{E} \bt \mc{E}^*[n]),
\ee
and therefore to give $\eta$ is equivalent to giving
\be
\tilde{\eta} \colon A \to H^n(X \times_S X, \mc{E} \bt \mc{E}^*),
\ee
i.e. equivalent to giving an element $\tilde{\eta}(1) \in H^n(X \times_S X, \mc{E} \bt \mc{E}^*)$.  We construct this element as follows.
Let $\Delta \colon X \to X \times_S X$ be the diagonal map, and note that we have a natural map
\be
H^n_{\Delta}(X \times_S X, \mc{E} \bt \mc{E}^*) \to H^n(X \times_S X, \mc{E} \bt \mc{E}^*)
\ee
where $H^n_{\Delta} = R^n H^0_\Delta$ is the cohomology with supports along $\Delta$.  Next, observe that for a sheaf $\mc{F}$ on $X \times_S X$, we have $H^0_{\Delta}(\mc{F}) := H^0(\Delta, \mc{H}^0_{\Delta}(\mc{F}))$, so local cohomology is the composition of global sections with taking the local cohomology sheaf.  
The corresponding Grothendieck spectral sequence reads
\be
E_2^{pq} \colon H^p(\Delta, \mc{H}^q_\Delta(\mc{F})) \implies H_{\Delta}^{p + q}(X \times_S X, \mc{F}).
\ee 

To describe these cohomology sheaves $\mc{H}^q_\Delta(\mc{F})$, we will use the following lemma. 

\begin{lemma}
\label{lem:coh}
Let $i \colon Z \to P$ be a Koszul-regular immersion of codimension $c$, and let $\mc{G}$ be a locally free sheaf on $P$.  Then $\mc{H}^a_Z(\mc{G})$ is 0 unless $a = c$.  Furthermore, the local sections of $\mc{H}^c_Z(\mc{G})$ can be described as the global sections of the cokernel of the final map in the extended alternating \v{C}ech complex associated to $i$ and $\mc{G}$, suitably localised (which will be explicitly described in the proof).
\end{lemma}

\begin{proof} 
The first half of this lemma is proven in \cite{stacks-project}*{\href{https://stacks.math.columbia.edu/tag/0G7P}{0G7P}}.
We explain the proof here, which also explains the second half of this lemma, 
namely how one can explicitly describe local sections of $\mc{H}_Z^c(\mc{G})$.  

We check these statements locally.
Suppose that $P = \Spec R$, $Z = V(f_1, \dots, f_c)$ with $f_1, \dots, f_c$ a Koszul-regular sequence of $R$.  Let $\mc{G} = \tilde{M}$, where $M$ is a locally free $R$-module.
Let $U_i = P - V(f_i)$, and $U_{i_0, \dots, i_k} := U_{i_0} \cap \dots \cap U_{i_k}$, which are all affine.
Let $j_{i_0, \dots, i_k} \colon U_{i_0, \dots, i_k} \to P$ be the inclusion and $\mc{G}_{i_0, \dots, i_k}$ be the restriction of $\mc{G}$ to $U_{i_0, \dots, i_k}$.  

Consider the sequence
\begin{equation}
    \label{eq: mod} 
    0 \to M \to \bigoplus_i M_{f_i} \to \bigoplus_{i < j} M_{f_i f_j} \to \dots \to M_{f_1 \dots f_c}
\end{equation}
and its associated sheafification 
\begin{equation}
    \label{eq: she}
0 \to \mc{G} \to \bigoplus_{i} (j_i)_* \mc{G}_i \to \dots \to (j_{1, \dots, c})_* \mc{G}_{1, \dots, c}, 
\end{equation}
where $\mc{G}$ is placed in degree 0.  By \cite{stacks-project}*{\href{https://stacks.math.columbia.edu/tag/0G7M}{0G7M}}, the complex (\ref{eq: she}) represents $i_*R\mc{H}_Z(\mc{G})$ in $D_Z(\O_P)$.  
Therefore the local cohomology sheaves vanish in degrees higher than $c$.  
Furthermore, by Koszul-regularity the cohomology is zero in degrees less than $c$ as well, proving the first part of the lemma. 

For the second part of the lemma, let $Q$ be the cokernel of the final map in the complex (\ref{eq: mod}), so that $\tilde{Q}$ is the cokernel of the final map in the complex (\ref{eq: she}).  Then again because the complex (\ref{eq: she}) represents $i_*R\mc{H}_Z(\mc{G})$, we see that the local sections of $\mc{H}^c_Z(\mc{G})$ are given by $\mc{H}^0_Z(\tilde{Q})$, as desired.
\end{proof}

\begin{corollary}
\label{corollary: cohg with support as cohg}
Let $i \colon Z \to P$ be a Koszul-regular immersion of codimension $c$, and let $\mc{G}$ be a locally free sheaf on $P$.  Then
\be
H^a_Z(P, \mc{G}) = H^{a-c}(Z, \mc{H}_Z^c(P, \mc{G})).
\ee
\end{corollary}

\begin{proof}
By \cref{lem:coh}, we have $R\Gamma_Z (P, \cG) = R\Gamma (Z, R \cH^0_Z (P, \cG)) = R\Gamma (Z, \cH^c_Z (P, \cG)[-c])$. Then taking $a$-th cohomology group implies the result.
\end{proof}
Since $X \to S$ is smooth, the diagonal map of $X$ is a Koszul-regular immersion \cite{stacks-project}*{\href{https://stacks.math.columbia.edu/tag/069G}{069G}}.
Thus by \cref{corollary: cohg with support as cohg}, to produce $\tilde{\eta}(1) \in H^n(X \times_S X, \mc{E} \bt \mc{E}^*)$ it suffices to produce a global section of the sheaf $\mc{H}^n_{\Delta}(X \times_S X, \mc{E} \bt \mc{E}^*)$, a task we can do locally in a neighborhood of $\Delta$ in $X \times_S X$ using the proof of Lemma \ref{lem:coh}. Indeed, the lemma implies that to give a local section of $\mc{H}^n_{\Delta}(X \times_S X, \mc{E} \bt \mc{E}^*)$, it suffices to give a section of $\tilde{Q}$, which we recall to be the cokernel of the final map in the complex (\ref{eq: she}).
Work in a local chart $U \subset X$ with coordinates $x_1, \dots, x_n$ giving an \'etale map to $\mb{A}^n_S$.  Let $x_1, \dots, x_n, y_1, \dots, y_n$ denote the corresponding coordinates on $U \times_S U \subset X \times_S X$.  

\begin{lemma}
\label{lem: diag}
The locus $V(x_1-y_1, \dots, x_n - y_n) \subset U \times_S U$ is the disjoint union of $\Delta_U$ and a closed subset of $U \times_S U$. In particular, there exists an open subset of $U \times_S U$ containing $\Delta_U$ on which $V(x_1 - y_1, \dots, x_n - y_n)$ cuts out precisely $\Delta_{U}$.
\end{lemma}

\begin{proof}
First, note that $V(x_1 - y_1, \dots, x_n - y_n)$ is just the image of $U \times_{\mb{A}^n_S} U$ under the natural closed embedding $U \times_{\mb{A}^n_S} U \to U \times_S U$, and moreover the diagonal morphism $U \to U \times_S U$ factors through the diagonal $U \to U \times_{\mb{A}^n_S} U$.
Since $U \to \mb{A}^n_S$ is \'etale, the diagonal map $U \to U \times_{\mb{A}^n_S} U$ is an open immersion \cite{stacks-project}*{\href{https://stacks.math.columbia.edu/tag/02GE}{02GE}}.
It follows that the complement of $\Delta_U$ in $U \times_{\mb{A}^n} U$ is a closed subset $V \subset V(x_1 - y_1, \dots, x_n - y_n)$, which was to be proven. 

Letting $g \in \Gamma(U \times_S U, \mc{O})$ be a function which is $1$ along $\Delta_U$ and $0$ along $V$,
we have that $V(x_1 - y_1, \dots, x_n - y_n) \cap D(g)$ is $\Delta_U$, as desired.
\end{proof}

Using this lemma, we may cover the diagonal of $X$ by open subsets $U'\subset \Delta_X$
which are cut out by $x_1 - y_1, \dots, x_n - y_n$, 
where $x_1, \dots, x_n, y_1, \dots, y_n$ are local coordinates yielding an \'etale map to $\mb{A}^n \times_S \mb{A}^n$.
Letting $e_1, \dots, e_r$ be a local basis  of $\mc{E}$, with $e_1^\vee, \dots, e_r^\vee$ denoting the dual basis in  $\mc{E}^\vee$, and letting $ \gw:= dy_1 \wedge \dots \wedge dy_n$ be a trivialisation of $\gw_{X/S}$, we can choose the section
\be
\tau_{\E}|_{U'} := \frac{\sum_{i = 1}^r e_i \bt (e_i^\vee \otimes \gw)}{(x_1 - y_1) \cdots (x_n - y_n)}.
\ee
It remains to show that the $\tau_{\E}|_{U'}$ glue to a global section on $X \times_S X$.
Let $\mathcal{I}$ be the quasicoherent sheaf of ideals of the diagonal $\Delta\colon X\rightarrow X\times_S X$.
We recall that there is a natural map \cite{stacks-project}*{\href{https://stacks.math.columbia.edu/tag/0G7T}{0G7T}}

\be
c \colon \wedge^n (\mathcal{I}/\mathcal{I}^2)^\vee \otimes_{\mc{O}_X} \Delta^*(\mc{E} \bt \mc{E}^*) \to \mc{H}^n_\Delta(\mc{E} \bt \mc{E}^*).
\ee
Note that changing a basis for $\mathcal{I}/\mathcal{I}^2$ alters the local map $\Delta^*(\mc{E} \bt \mc{E}) \to \mc{H}_{\Delta}^n(\mc{E} \bt \mc{E})$ by a factor of a determinant, so the maps will glue if we include the factor of $\wedge^n (\mathcal{I}/\mathcal{I}^2)^\vee$. We have canonical isomorphisms
\be
\wedge^n (\mathcal{I}/\mathcal{I}^2)^\vee \otimes_{\mc{O}_X} \Delta^*(\mc{E} \bt \mc{E}^*) \cong \gw_{X/S}^\vee \otimes\mc{E} \otimes \mc{E}^\vee \otimes \gw_{X/S} \cong \mc{E}nd(\mc{E}),
\ee
and thus can consider the global section
\be
c(\id_{\mc{E}}) \in H^0(X \times_S X, \mc{H}^n_{\Delta}(\mc{E} \bt \mc{E}^*)).
\ee
Over the open subset $U'$ constructed above, we can trace through the above identifications and find that
\be
c(\id_\mc{E})|_{U'} = c((dy_1 \wedge \dots \wedge dy_n)^\vee \otimes (\sum_i e_i \otimes e_i^\vee) \otimes dy_1 \wedge \dots \wedge  dy_n) = \frac{\sum_{i = 1}^r e_i \boxtimes (e_i^\vee \otimes \gw)}{(x_1 - y_1)\dots(x_n - y_n)} = \tau_{U'},
\ee
so indeed $c(\id_\mc{E})$ yields the desired global section, and hence the desired element $\tilde{\eta}(1) \in H^n(X \times_S X, \mc{E} \bt \mc{E}^*)$.

\section{Compatibility with Serre duality} 
\label{sec: compatibility}
\subsection{Setup}
Using the notation from the previous section, recall that Serre duality gives a perfect pairing 
\be
\ep: R \Gamma(X, \mc{E}^*)[n] \tten_A R \Gamma(X, \mc{E}) \to A,
\ee
of perfect complexes in $D(A)$.  Now that we have constructed a map 
\[\eta\colon A\to \rg(X^2, \ebe[n])\simeq \rg(X,\mc{E})\otimes_A^{\mb{L}}\rg(X,\mc{E}^*)[n]\] 
by constructing a distinguished element $\eta(1)=\tau_{\E}$, 
we will show that it gives a coevaluation to the evaluation $\ep$.  That is, we will check that the following composition is the identity:
\begin{equation}
\label{eqn: comp}
    \rg(X,\mc{E})\xrightarrow{\eta\otimes id} \rg(X,\mc{E})\otimes^{\mb{L}}_A\rg(X,\mc{E}^*)[n]\otimes^{\mb{L}}_A \rg(X,\mc{E})\xrightarrow{id\otimes \epsilon}  R\Gamma(X,\mc{E}).
\end{equation}

\subsection{Factoring the composition} 
Computing the composition directly can be tricky; therefore we will factor it in a way that allows us to use the local description of $\eta$.

First we note that there's a map $R\Gamma(X,\mc{E})\xrightarrow{p_{1}^*}\rg(X^3,\E\boxtimes \mc{O}_X\boxtimes \mc{O}_X)$ which is pullback by the projection on the first factor. 
Now recall that $\tau_{\E}\in H^n_{\Delta}(X^2,\ebe)$ which we pullback along the $p_{23}\colon X^3\to X^2$ to get a class $p^{*}_{23}\tau_{\E}\in H^nR\Gamma_{\Delta\times X}(X^3, \O_X\boxtimes \E^*\boxtimes \E)$.  By taking the cup product with this class, we have a sequence of maps 

$$R\Gamma(X,\mc{E})\xrightarrow{p_3^*}R\Gamma(X^3,\mc{O}_X\boxtimes \mc{O}_X\boxtimes \mc{E})\xrightarrow{(-)\cup p_{12}^*\tau} R\Gamma_{\Delta\times X}(X^3,\mc{E}\boxtimes \mc{E}^*[n]\boxtimes \mc{E})\to R\Gamma(X^3,\mc{E}\boxtimes \mc{E}^*[n]\boxtimes \mc{E}).$$

Pulling back along $\id\times \Delta\colon X\times X\rightarrow X^3$, we get a map $R\Gamma(X^3,\mc{E}\boxtimes \mc{E}^*[n]\boxtimes \mc{E})\xrightarrow{(id\times \Delta^*)} R\Gamma(X^2,\mc{E}\boxtimes (\mc{E}^*\otimes \mc{E})[n])$ and via the trace map we get a map 

\begin{equation}
   \label{eqn: trace}
   R\Gamma(X^2,\mc{E}\boxtimes (\mc{E}^*\otimes \mc{E})[n])\xrightarrow{\mathrm{Tr}} R\Gamma(X^2,\mc{E}\boxtimes (\omega_{X/S})[n])\to R\Gamma(X,\mc{E})\otimes^{\mb{L}}_A R\Gamma(X,\omega_{X/S}[n])\xrightarrow{1\otimes \int_X} R\Gamma(X, \mc{E}).
\end{equation}

We claim that the composition of these maps just described is equal to the composition in \eqref{eqn: comp}.  Indeed, this is equivalent to the commutativity of the following diagram. \\ 

\adjustbox{scale=0.75,center}{%
\begin{tikzcd}
	{R\Gamma(X^3, \O_X\boxtimes \O_X\boxtimes \mc{E})} & {R\Gamma_{\Delta\times X}(X^3, \mc{E}\boxtimes \mc{E}^*[n]\boxtimes \mc{E})} & {R\Gamma(X^3, \mc{E}\boxtimes \mc{E}^*[n]\boxtimes \mc{E})} & {R\Gamma(X^2, \mc{E}\boxtimes(\mc{E}^*[n]\otimes\mc{E}))} \\
	{R\Gamma(X, \mc{E})} && {R\Gamma(X, \mc{E})\tten_AR\Gamma(X, \mc{E}^*)[n]\tten_A R\Gamma(X, \mc{E})} & {R\Gamma(X, \mc{E})}
	\arrow["(-)\cup {p_{12}^*\tau_{\E}}", from=1-1, to=1-2]
	\arrow[from=1-2, to=1-3]
	\arrow["{(\id \times \Delta)^*}", from=1-3, to=1-4]
	\arrow["{\eta\otimes \id}"', from=2-1, to=2-3]
	\arrow["{p_3^*}", from=2-1, to=1-1]
	\arrow["{\text{K\"unneth}}"', from=2-3, to=1-3]
	\arrow["{\id\times \epsilon}"', from=2-3, to=2-4]
	\arrow["{\Tr, \text{K\"unneth}, \epsilon}", from=1-4, to=2-4]
\end{tikzcd}
}

Here the rightmost map denoted `$\Tr, \text{K\"unneth}, \epsilon$' is the composition described in \eqref{eqn: trace}. This diagram commutes because each of the two squares commutes by construction.  Thus we need to show that the composition along the top side of the rectangle is the identity. \\ 

Note that the composition from the bottom left corner to the top right corner can be simplified to 
\[
R\Gamma(X, \E)\xrightarrow{p_2^*}R\Gamma(X^2, \O_X\boxtimes\E)\xrightarrow{\cup \tau_{\E}} R\Gamma(X^2, \E\boxtimes(\E^*[n]\otimes \E)).
\]
Furthermore, since $\tau_{\mc{E}}\in H^0(X \times_S X, \mc{H}^n_{\Delta}(\mc{E} \bt \mc{E}^*))$ is supported on the diagonal, these maps also factor through the corresponding local cohomology groups.  Thus checking that the composition \eqref{eqn: trace} is the identity is equivalent to the following result.  
\begin{proposition}
\label{prop: id1}
    The composition
    \begin{align*}
    R\Gamma(X, \mc{E}) & \xrightarrow{p_2^*}R\Gamma(X^2, \O_X\boxtimes\E) \xrightarrow{\cup\tau_{\E}} {R\Gamma_{\Delta}(X^2, \mc{E}\boxtimes (\mc{E}\otimes \mc{E}^*))} \xrightarrow{\Tr} R\Gamma_{\Delta}(X^2, \mc{E}\boxtimes \omega_{X/S}[n]) \\ 
    &\rightarrow R\Gamma(X^2, \mc{E}\boxtimes \omega_{X/S}[n])\xrightarrow{\text{K\"unneth}, \epsilon} R\Gamma(X, \mc{E})
\end{align*} 
is the identity. 
\end{proposition}

\subsection{\v Cech check} 
We will now prove Proposition \ref{prop: id1} using \v Cech cohomology.
Let $\mc{U}$ be an open affine covering of $X$.  Because $X$ is separated, we can identify $R\Gamma(X, \mc{E})= C^*(\mc{U}, \mc{E})$ and $R\Gamma(X\times X, \O_X\boxtimes\mc{E})= C^*(\mc{U}\times \mc{U}, \O_X\times \mc{E})$.  After intersecting each $U_i\times U_j$ with the diagonal, we obtain an open affine covering $\mc{U}'$ of $X\times X$. 

We will compute the composition in Proposition \ref{prop: id1} up until the final morphism, given by the top row in the commutative diagram below, using the identifications made in the bottom row. 

\adjustbox{scale=0.9}{
\begin{tikzcd}
	{R\Gamma(X, \mc{E})} & {R\Gamma(X\times X, \O_X\boxtimes\mc{E})} & {R\Gamma_{\Delta}(X^2, \mc{E}\boxtimes (\mc{E}\otimes \mc{E}^*))} & {R\Gamma_{\Delta}(X^2, \mc{E}\boxtimes\omega_{X/S})} \\
	{C^*(\mc{U}, \mc{E})} & {C^*(\mc{U}\times \mc{U}, \O_X\boxtimes \mc{E})} & {C^*(\mc{U}\times \mc{U}, \mc{H}^n_{\Delta}(\mc{E}\boxtimes (\mc{E}\otimes \mc{E}^*)))} & {C^*(\mc{U}', \mc{H}^n_{\Delta}(\mc{E}\boxtimes \omega_{X/S}))} \\
	& {C^*(\mc{U}_{\alpha}\times \mc{U}_{\alpha}, \mc{H}^n_{\Delta}(\mc{E}\boxtimes (\mc{E}\otimes \mc{E}^*)))} & {C^*(\mc{U}', \mc{H}^n_{\Delta}(\mc{E}\boxtimes (\mc{E}\otimes \mc{E}^*)))}
	\arrow["{p_2^*}", from=1-1, to=1-2]
	\arrow["{\cup \tau_E}", from=1-2, to=1-3]
	\arrow["\Tr", from=1-3, to=1-4]
	\arrow["{p_2^*}"', from=2-1, to=2-2]
	\arrow["{\cup \tau_E}"', from=2-2, to=2-3]
	\arrow[from=1-1, to=2-1, equal]
	\arrow[from=1-2, to=2-2, equal]
	\arrow[from=1-3, to=2-3, equal]
	\arrow[from=2-3, to=2-4]
	\arrow["\cong"', from=3-2, to=2-3]
	\arrow["\cong"', from=3-2, to=3-3]
	\arrow[from=2-3, to=3-3, equal]
	\arrow["\Tr"', from=3-3, to=2-4]
	\arrow[from=1-4, to=2-4, equal]
\end{tikzcd}
}

Indeed, recall that in the Grothendieck spectral sequence associated to the composition $R\Gamma(\Delta, R\mc{H}_{\Delta}(\mc{E}\boxtimes (\mc{E}^*\otimes \mc{E})) = R\Gamma_{\Delta}(X\times X, \mc{E}\boxtimes (\mc{E}^*\otimes \mc{E}))$ the only nonzero degree of $\mc{H}^i_{\Delta}(\mc{E}\boxtimes (\mc{E}^*\otimes \mc{E})$ occurs when $i=n$.  Thus $R\Gamma_{\Delta}(X\times X, \mc{E}\boxtimes (\mc{E}^*\otimes \mc{E}))=C^*(\mc{U}', \mc{H}^n_{\Delta}(\mc{E}\boxtimes (\mc{E}\otimes \mc{E}^*)))$. \\ 

Now recall that $\tau_{\mc{E}}\in H^0(X \times_S X, \mc{H}^n_{\Delta}(\mc{E} \bt \mc{E}^*))$.  Thus taking the cup product with it lands in ${C^*(\mc{U}\times \mc{U}, \mc{H}^n_{\Delta}(\mc{E}\boxtimes (\mc{E}\otimes \mc{E}^*)))}$.  Since the sheaf is supported on the diagonal, we may compute its cohomology by restricting to the elements of the open cover $\mc{U}\times \mc{U}$ of the form $\mc{U}_\alpha \times \mc{U}_\alpha$.  Thus the natural inclusion morphism of \v Cech complexes ${C^*(\mc{U}_\alpha\times \mc{U}_\alpha, \mc{H}^n_{\Delta}(\mc{E}\boxtimes (\mc{E}\otimes \mc{E}^*)))}\rightarrow {C^*(\mc{U}\times \mc{U}, \mc{H}^n_{\Delta}(\mc{E}\boxtimes (\mc{E}\otimes \mc{E}^*)))}$ is an isomorphism in the derived category.  Restricting these covers to $\mc{U}'$ gives the identification with ${C^*(\mc{U}', \mc{H}^n_{\Delta}(\mc{E}\boxtimes (\mc{E}\otimes \mc{E}^*)))}$. \\ 

Thus, we wish to explicitly compute the value of the composition of the horisontal maps on an element $s\in C^*(\mc{U}, \mc{E})$ by evaluating along the bottom row.  We have $p_2^*s$ is the element 
\[
1\boxtimes s\in C^*(\mc{U}\times \mc{U}, \O_X\times \mc{E}), \qquad (1\boxtimes s)(\mc{U}_\alpha\times \mc{U}_\beta) = s(\mc{U}_{\beta}), 
\]
where $\alpha$ and $\beta$ are $p$-tuples of $s$ is in degree $p$.  Now we must compute its value when taking the cup product with $\tau_{\mc{E}}$, which on the open subsets $\mc{U}'$ is given by 
\[
\tau_{\mc{E}}|_{\mc{U}'} = \frac{\sum_{i = 1}^r e_i \bt (e_i^\vee \otimes \gw)}{(x_1 - y_1)\dots (x_n - y_n)}  
\]
By the earlier discussion, we can compute the cup product with this element by restricting to the values on $\mc{U}'$.  In doing so, the value the cocycle $s$ takes on $\mc{U}'_{\alpha} = (U_\beta\times U_\gamma)\cap \Delta$ is given by $s(U_\beta\times U_\gamma)|_{\mc{U}'_\alpha}$, which we will simply write as $s(\mc{U}'_{\alpha})$.   We have that $p_2^*s\cup \tau_E$ evaluated on $\mc{U}'_\alpha$ is 

\begin{align*}
(p_2^*s\cup \tau_E)(\mc{U}'_\alpha) 
&= \dfrac{(\sum_{i = 1}^r e_i \bt (e_i^\vee \otimes \gw))\otimes (1\boxtimes s)}{(x_1-y_1)\cdots (x_n-y_n)}(\mc{U}'_\alpha) \\ 
&= \dfrac{(\sum_{i = 1}^r e_i \bt (e_i^\vee \otimes \gw\otimes s ))}{(x_1-y_1)\cdots (x_n-y_n)}(\mc{U}'_\alpha)
\end{align*}

Applying the trace map, we get 

\begin{align*}
(\Tr (p_2^*s\cup \tau_E))(\mc{U}'_\alpha) &= \dfrac{\sum_{i = 1}^r (e_i \bt e_i^\vee(s) \gw)}{(x_1-y_1)\cdots (x_n-y_n)}(\mc{U}'_\alpha) \\ 
&= \dfrac{\sum_{i = 1}^r (e_i^\vee(s) e_i \bt \gw)}{(x_1-y_1)\cdots (x_n-y_n)}(\mc{U}'_\alpha) \\ 
&= \dfrac{s \bt \gw}{(x_1-y_1)\cdots (x_n-y_n)}(\mc{U}'_\alpha). 
\end{align*}
This shows that if we write factor the desired composition in Proposition \ref{prop: id1} as 
\[
R\Gamma(X, \mc{E}) \xrightarrow{g} R\Gamma(X^2, \mc{E}\boxtimes \omega_{X/S}[n])\xrightarrow{\text{K\"unneth}, \epsilon} R\Gamma(X, \mc{E}), 
\]
we have that the section $g(s)$ is locally represented by $\dfrac{s \bt \gw}{(x_1-y_1)\cdots (x_n-y_n)}$ on $\mc{U}'_\alpha$.  We recall that the term $\frac{1\bt \omega}{(x_1-y_1)\cdots (x_n-y_n)}$ represents a section of $R\Gamma(X^2, \O_X\boxtimes \omega_{X/S})$ by giving a \v{C}ech cocycle representing a global section of the local cohomology sheaf, under the isomorphism 
\[
 H^0(X \times_S X, \mc{H}^n_{\Delta}(\O_X \bt \omega_{X/S})) \cong H^n(X \times_S X, \O_X \bt \omega_{X/S}). 
\]
By \cite[\href{https://stacks.math.columbia.edu/tag/0G7Q}{Tag 0G7Q}]{stacks-project}, the section of $\mc{H}^n_{\Delta}(\O_X \bt \omega_{X/S})$ that $\frac{1\bt \omega}{(x_1-y_1)\cdots (x_n-y_n)}$ represents corresponds to $1\bt \omega$ as a section of $H^0(\mc{U}'_\alpha\times_S \mc{U}'_\alpha, \O_X|_{\mc{U}'_\alpha}\boxtimes\omega_{X/S}|_{\mc{U}'_\alpha}[n])$.  Moreover, by the construction of $\omega$ in Section \ref{sec: copairing locally free}, $\omega$ is the restriction to each $\mc{U}'$ of the element sent to 1 by the trace map $\eps: H^n(X, \omega_{X/S})\rightarrow A$.  Thus applying K\"unneth and $\eps$ to $g(s)$ yields $s$, as desired.  This completes the proof of Proposition \ref{prop: id1}.

\section{Extension to two-term complexes}
\label{sec: two-term}
\subsection{Definition of the copairing}
Continuing with the same notation as in the previous sections, let $\E$ and $\F$ be locally free sheaves on $X$.  Set $L = \mc{E} \overset{D}\to \mc{F}$ with $\E$ in degree 0, and $M = L^* = \mc{F}^* \overset{D^*}\to \mc{E}^*$ with $\E^*$ in degree 0.  By the filtered K\"unneth formula, to construct the copairing map $\eta: A \to R\Gamma(X, L) \tten_A R\Gamma(X, M)[n]$, it is enough to give a distinguished element
\be
\tilde{\eta}(1) \in H^n(X \times_S X, L \bt M).
\ee

The following proposition lets us produce this distinguished element in the same way as we did for 1 term complexes before.

\begin{proposition}
\label{prop: hyperrep}
Consider the sequence
\be
 \cdots\rightarrow 0\rightarrow H^n_{\Delta}(\mc{E} \bt \mc{F}^*) \to H^n_{\Delta}(\mc{E} \bt \mc{E}^*) \oplus H^n_{\Delta}(\mc{F} \bt \mc{F}^*) \to H^n_\Delta(\mc{F} \bt \mc{E}^*)\rightarrow 0\rightarrow\cdots 
\ee
where the middle term is placed in degree $n$ and the last map is induced by $D \bt 1 - 1 \bt D^*$, arising from the first spectral sequence of hypercohomology abutting to 
$R\Gamma_\Delta(X \times_S X, L \bt M)$.  Then an element of $H^n_{\Delta}(\mc{E} \bt \mc{E}^*) \oplus H^n_{\Delta}(\mc{F} \bt \mc{F}^*)$ in the kernel of $D \bt 1 - 1 \bt D^*$ naturally defines an element of $H^n(R\Gamma_\Delta(X \times_S X, L \bt M))$.
\end{proposition}

\begin{proof}
Let $\mc{G}$ denote one of the locally free sheaves $\mc{E}\boxtimes \mc{E}^*, \mc{E}\boxtimes \mc{F}^*, \mc{F}\boxtimes\E^*, \F\boxtimes \F^*$. 
Note that the $E_2$-spectral sequence $E_2^{a, b} = H^a(X\times X, \mc{H}^b_{\Delta}(\mc{G}))\Rightarrow H^{a+b}_{\Delta}(X\times X, \mc{G})$ degenerates with only nonzero term when $b=n$, by Lemma \ref{lem:coh}.  In particular, $H^{a+b}_{\Delta}(X\times X, \mc{G})$ is only nonzero for $a+b\ge n$.  

Next, observe that
\be
L \bt M \cong \mc{E} \bt \mc{F}^* \to (\mc{E} \bt \mc{E}^*) \oplus (\mc{F} \bt \mc{F}^*) \overset{D \bt 1 - 1 \bt D^*}\to \mc{F} \bt \mc{E}^*
\ee
where the middle term is in degree 0.
Then the first spectral sequence of hypercohomology for $R\Gamma_\Delta(X \times_S X, L \bt M)$ gives an $E_1$-spectral sequence consisting of terms of the form $E_1^{a, b} = H^{b}_{\Delta}(X\times X, \mc{G})$, which are only nonzero for $b\ge n$.  Thus any element of $H^n_{\Delta}(\mc{E} \bt \mc{E}^*) \oplus H^n_{\Delta}(\mc{F} \bt \mc{F}^*)$ in the kernel of $D \bt 1 - 1 \bt D^*$ is also in the kernel of all differentials from that position on subsequent pages of the spectral sequence, as such differentials go downwards.  This proves the desired result. 
\end{proof}

An element of $H^n(R\Gamma_{\Delta}(X\times_S X, L\boxtimes M))$ gives an element of $H^n(X\times_S X, L\boxtimes M)$ by the natural map between them.  
Thus to define a distinguished element $\tilde{\eta}(1) \in H^n(X \times_S X, L \bt M)$, it suffices to give elements of $H^n_\Delta(\mc{E} \bt \mc{E}^*)$ and $H^n_{\Delta}(\mc{F} \bt \mc{F}^*)$ which are killed by $D \bt 1 - 1 \bt D^*$.  
For these, we will take the elements $\tau_{\mc{E}}$ and $\tau_{\mc{F}}$ constructed in Section \ref{sec: copairing locally free} and let $\tilde{\eta}(1)$ correspond to $\tau_{\mc{E}} + \tau_{\mc{F}}$. 

\begin{construction}
    [Construction for two-term complexes] Define the copairing map \be
\eta: A \to R\Gamma(X, L) \tten_A R\Gamma(X, M).
\ee
by setting $
\tilde{\eta}(1) \in H^n(X \times_S X, L \bt M)$ to be the image of $\tau_{\E}+\tau_{\F}\in H^n_{\Delta}(X\times_S X, L\boxtimes M)$ in $H^n(X\times_S X, L\boxtimes M)$.
\end{construction} 

To show that this construction is well-defined, we must check that $\tau_{\E}+\tau_{\mc{F}}$ is indeed a cocycle in the representation of $R\Gamma_\Delta(X \times_S X, L \bt M)$ given by Proposition \ref{prop: hyperrep}.  We recall that $\tau_{\E}$ and $\tau_{\F}$ are defined in the discussion following Lemma \ref{lem: diag}, which we briefly review now.  In this definition, we first look at an open covering of $\Delta_X$ by $\Delta_U$ over which $\mc{E}$ and $\mc{F}$ are locally free, giving coordinates $x_1, \ldots, x_n, y-1, \ldots, y_n$.  Then restrict to a smaller open covering by $\Delta_{U'}$ over which $V(x_1-y_1, \ldots x_n-y_n) = \Delta_U$.  

\subsection{Proof of well-definedness}
Because \'etale morphisms are locally standard \'etale, we may assume that $U$ is standard \'etale over $S=\Spec A$, which will help in our explicit calculation.   

\begin{proposition}
\label{proposition: equality}
    Work in a local chart $U \subset X$ over which $\mc{E}$ and $\mc{F}$ are free with coordinates $x_1, \dots, x_n$ giving a standard \'etale map to $\mathbb{A}^n_S$.  Let $x_1, \dots, x_n, y_1, \dots, y_n$ denote the corresponding coordinates on $U \times_S U \subset X \times_S X$. Letting $e_1, \dots, e_r$ be a local basis of $\mathcal{E}$ and $f_1,\dots f_t$ be a local basis of $\mathcal{F}$ and $ \omega:= dy_1 \wedge \dots \wedge dy_n$ be a trivialisation of $\omega_{X/S}$.  As in Lemma \ref{lem: diag}, let $g\in \Gamma(U\times_S U, \O)$ be a function such that $U'=D(g)$ contains $\Delta_U$ and $V(x_1-y_1, \ldots, x_n-y_n) = \Delta_U$.  We have the elements $\tau_{\E}, \tau_{\F}$ of $H^n_\Delta(\mc{E} \bt \mc{E}^*)$ and $H^n_{\Delta}(\mc{F} \bt \mc{F}^*)$ respectively defined locally by
\begin{align*}
\tau_{\mathcal{E}}|_{U'} &= \frac{\sum_{i = 1}^r e_i \boxtimes (e_i^\vee \otimes \omega)}{(x_1 - y_1)\dots (x_n - y_n)} \in \Gamma(U', \mc{H}^n_{\Delta}(\E\boxtimes \E^*)) \\
\tau_{\mathcal{F}}|_{U'} &= \frac{\sum_{i = 1}^t f_i \boxtimes (f_i^\vee \otimes \omega)}{(x_1 - y_1)\dots (x_n - y_n)} \in \Gamma(U', \mc{H}^n_{\Delta}(\F\boxtimes \F^*)) .
\end{align*}
Then $(D \boxtimes 1) (\tau_{\mathcal{E}}) = (1 \boxtimes D^*)(\tau_{\mathcal{F}})$ in $H^n_{\Delta}(X \times_S X, \mathcal{F}^* \boxtimes \mathcal{E})$.
\end{proposition}

Before we proceed with the proof, we will make explicit what we need to prove in terms of algebra.  Since $U=\Spec R$ is standard \'etale over $\mb{A}^n_S$, where $S=\Spec A$, by definition we can write $R = (A[x_1, \ldots, x_n][\alpha]/f_x)_{h_x}$ where $f_x, h_x\in A[x_1, \ldots, x_n][\alpha]$, $f_x(\alpha)$ is monic, and $f_x'(\alpha)$ is invertible in $R$.  Then $U' = \Spec (A[x_1, \ldots, x_n][\alpha]/f_x)_{gh_x}$ 

It suffices to prove the case where $\mc{E}$ and $\mc{F}$ are both rank 1. 
 Given an $A$-linear differential operator $D$ on $R$, let $D_x$ and $D_y$ be the $A$-linear differential operators on $R\otimes_A R$ defined by $D_x(r\otimes r') = D(r)\otimes r'$ and $D_y(r\otimes r') = r\otimes D(r')$.  
We wish to check that $(D\boxtimes 1)(\tau_{\E}|_{U'}) - (1\boxtimes D^*)(\tau_{\F}|_{U'}) = 0$.  
The explicit representative we are given for this element is in terms of the extended alternating \v{C}ech complex.  
Therefore, if we let $\xi = ((x_1 - y_1) ... (x_n - y_n))^{-1}\in (R\otimes_A R)_{g(x_1-y_1)\cdots (x_n-y_n)}$, then it suffices to show that 
\[
D_x(\xi) - D_y^*(\xi) \in \im\left( \bigoplus_{i=1}^n (R\otimes_A R)_{g(x_1-y_1)\cdots \widehat{(x_i-y_i)}\cdots (x_n-y_n)}\rightarrow (R\otimes_A R)_{g(x_1-y_1)\cdots (x_n-y_n)} \right). 
\]

\begin{proof}
[Proof of Proposition \ref{proposition: equality}] 
The proof proceeds in two steps.  First we show the result holds for a class of finite order differential operators which generate all finite order differential operators as an algebra.  Then we show that if the result holds for two differential operators $D_1$ and $D_2$, then it holds for the composition $D = D_1D_2$.  

Recall that a finite order differential operator can be written as a sum of terms of the form $r\Big( \frac{1}{\beta!}\cdot \frac{\partial}{\partial x_i^\beta}\Big)$, with $r\in R$.  Thus for the first step it suffices to prove the case of $D_x=\frac{\partial}{\partial x_i}$, the case $D_x = \frac{1}{(p^k)!}\frac{\partial}{\partial x_1^{p^k}}$ (this case is only relevant in characteristic $p$), and finally, the case of 0th order differential operators given by some element $p_x(\alpha)\in R$ where $p$ is a one-variable polynomial with coefficients in $A[x_1, \ldots, x_n]_{gh_x}$. 

For the first step, if $D_x = \frac{\partial}{\partial x_i}$ then $D_y^* = -\frac{\partial}{\partial y_i}$, and 
\[
\frac{\partial}{\partial x_i}\left(\frac{1}{(x_1-y_1)\cdots (x_n-y_n)}\right) - \left(-\frac{\partial}{\partial y_i}\left(\frac{1}{(x_1-y_1)\cdots (x_n-y_n)}\right)\right) = 0.
\]

Second, for $D_x = \frac{1}{(p^k)!}\frac{\partial}{\partial x_i^{p^k}}$, we have $D_y = -\frac{1}{(p^k)!}\frac{\partial}{\partial y_i^{p^k}}$ if $p$ is odd and $D_y = \frac{1}{(p^k)!}\frac{\partial}{\partial y_i^{p^k}}$ if $p$ is 2.  In either case, since the action of $D_x$ and $D_y^*$ on $\xi$ is the same as one would get by formally taking derivatives as usual, a simple computation shows that $D_x(\xi)-D_y^*(\xi)=0$.

 In the 0th order case, we have $D_x(\xi) - D_y(\xi) = (p_x(\alpha)-p_y(\beta))\xi$, where $p_x$ is a single-variable polynomials with coefficients in $A[x_1, \ldots, x_n]_{h_x}$ and $p_y$ is the same except with $x_i$ replaced with $y_i$ and $h_x$ replaced with $h_y$.  The fact that $V(x_1-y_1, \ldots, x_n-y_n)=\Delta_U$ implies that $\alpha-\beta\in (x_1-y_1, \ldots, x_n-y_n)$.  Then setting $x_i=y_i$ and $\alpha = \beta$, we obtain $p_x(\alpha)-p_y(\beta)=0$; this means that $p_x(\alpha)-p_y(\beta)\in (x_1-y_1, \ldots, x_n-y_n)\subset (R\otimes_A R)_g$, and the desired result follows.

Now for the second part, we will show that if the result holds for a pair of differential operators $D_1, D_2$, then it holds for their composition $D = D_1D_2$.  We use the following two facts.  
\begin{enumerate}
    \item Given any differential operator $T\colon R\otimes_A R \rightarrow R\otimes_A R$, any element $h\in R\otimes_A R$ induces a differential operator $T\colon (R\otimes_A R)_h \rightarrow (R\otimes_A R)_h$. (This follows from \cite[\href{https://stacks.math.columbia.edu/tag/0G44}{Tag 0G44}]{stacks-project}.) 

    \item For any pair of $A$-linear differential operators $D', D''$ on $R$ we
have that $D'_x$ and $D''_y$ commute.
\end{enumerate}

We know that

\[D_{1, x}( \xi ) = D_{1, y}^*( \xi ) + \sum E_{1, i}, \quad D_{2, x}( \xi ) = D_{2, y}^*( \xi ) + \sum E_{2, i}\]

where $E_{1, i}, E_{2, i}\in (R\otimes_A R)_{g(x_1-y_1)\cdots \widehat{(x_i-y_i)}\cdots (x_n-y_n)}$.

Now we use the second fact to compute the following.

\begin{align*}
  D_x( \xi )
= D_{1, x}( D_{2, x}( \xi ) )
&= D_{1, x}( D_{2, y}^*(\xi) + \sum E_{2, i} ) \\
&= D_{1, x}( D_{2, y}^*(\xi) ) + \sum D_{1, x}( E_{2, i} ) \\
&= D_{2, y}^*( D_{1, x}( \xi ) ) + \sum D_{1, x}( E_{2, i} ) \\
&= D_{2, y}^*( D_{1, y}^*( \xi ) + \sum E_{1, i} ) + \sum D_{1, x}( E_{2, i} ) \\ 
&= D_{2, y}^*( D_{1, y}^*( \xi ) ) + D_{2, y}^*( \sum E_{1, i} ) + \sum
D_{1, x}( E_{2, i} ) \\
&= D_y^*( \xi ) + \sum D_{2, y}^*( E_{1, i} ) + \sum D_{1, x}( E_{2, i} ).  
\end{align*}
Thus by the first fact, $D_x(\xi)-D_y(\xi)$ is indeed an element of $\bigoplus_{i=1}^n (R\otimes_A R)_{g(x_1-y_1)\cdots \widehat{(x_i-y_i)}\cdots (x_n-y_n)}$, as desired.  This completes the proof of the proposition.
\end{proof}

\subsection{Recollections on duality in the filtered derived category}
In order to properly describe how the constructed copairing is dual to the one given by Serre duality, we will use the language of filtered derived categories.  We thus recall the basic notions of duality in the filtered derived category. 

Given a ring $A$, the filtered derived category $DF(A)$ arises as the localisation of the category of filtered chain complexes at graded quasi-isomorphisms, i.e., morphisms which induce an isomorphisms at the graded level.  The filtered derived category has an internal Hom functor 
\begin{equation}
    R\underline{\mathrm{Hom}}(-,-)\colon DF(A)\times DF(A)\to DF(A)
\end{equation}
 and a symmetric monoidal filtered tensor product 
\begin{equation}
\label{tensor}
    -\underline{\otimes}^L-\colon DF(A)\to DF(A).
\end{equation}

These functors are adjoints and the following formulas hold:

\begin{equation}\label{global adjunction}
    \mathrm{Hom}_{DF(A)}(M\underline{\otimes}^LN,P)=\mathrm{Hom}_{DF(A)}(M,R\underline{\mathrm{Hom}}(M,P)),
\end{equation}
\begin{equation}\label{local adjunction}
    R\underline{\mathrm{Hom}}(M\underline{\otimes}^LN,P)=R\underline{\mathrm{Hom}}(M,R\underline{\mathrm{Hom}}(M,P)).
\end{equation}

The functor \ref{tensor} makes $(DF(A),\underline{\otimes}^L)$ into a symmetric monoidal category and therefore we can talk about the notion of dualisable objects in $(DF(A),\underline{\otimes}^L)$. 

We recall here that in a symmetric monoidal category $(\mathcal{C},\otimes)$ an object $A$ is dualisable with dual $A^\vee$ if there is an \emph{evaluation} morphism $\epsilon\colon A\otimes A^\vee \to  \mathbf{1}_{\cC}$ and a \emph{coevaluation} morphism $\eta \colon  \mathbf{1}_{\cC}\to  A\otimes A^\vee$ which together satisfy certain compatibility conditions.

\begin{example}
    Let $D(A)$ be the usual derived category then an object $P\in D(A)$ is dualisable precisely when it is perfect i.e. quasi-isomorphic to a complex of finite projective modules over $A$. In fact when $P$ is perfect, then it's dual is given by $R\mathrm{Hom}(-,A)$, the coevaluation map $\eta\colon A\to R\mathrm{Hom}(P,A)\otimes^L P$ is the trace, while the evaluation $R\mathrm{Hom}(P,A)\otimes^L P\to A$ is the usual evaluation map, arising from the counit of the adjunction $$\mathrm{Hom}_{D(A)}(R\mathrm{Hom}(P,A)\otimes^LP,A)=\mathrm{Hom}_{D(A)}(R\mathrm{Hom}(P,A),R\mathrm{Hom}(P,A)).$$ 
\end{example}

In $DF(A)$ we thus want to isolate a notion of perfect complexes which interacts well with the notion of perfectness in $D(A).$ 
Note that dualisability is a diagramatic notion and symmetric monoidal functors send dualisable objects to dualisable objects. The functor $\mathrm{Forget}\colon DF(A)\to D(A)$ is symmetric monoidal and so dualisable objects in $DF(A)$ should remain dualisable after forgetting the filtration. 
Similar considerations with the symmetric monoidal functors $\gr_i\colon DF(A)\to D(A)$ and $\gr\colon DF(A)\to DG(A)\xrightarrow{\mathrm{Forget}} D(A)$ (where $DG(A)$ is the graded derived category) show that if $P\in DF(A)$ is perfect, then $\gr_iP\in D(A)$ should be perfect for each $i$ and that $\gr_iP$ can only be non-zero in finitely many degrees.

From this we obtain the following characterisation of dualisable objects.

\begin{proposition}\label{dualisable characterisation}
    The following are equivalent
    \begin{enumerate}
        \item $P\in DF(A)$ is dualisable in $(DF(A),\underline{\otimes}^L),$
\item $P$ has a finite filtration (i.e. $\gr_iP$ is non-zero for finitely many $i$) and all the $\gr_iP$ are perfect in $D(A)$,
\item $P\in DF(A)$ is dualisable in $(DF(A),\underline{\otimes}^L)$ with dual $R\underline{\mathrm{Hom}}(P,A\{0\}).$
    \end{enumerate}
\end{proposition}

\begin{definition}\label{definition of perfect complexes}
    A complex $P\in DF(A)$ is called perfect if it satisfies the equivalent properties in Proposition \ref{dualisable characterisation}.
\end{definition}

\subsection{Compatibility with filtered Serre duality} 
Returning to the specific situation of this paper, using the Hodge filtration, the copairing map 
\[
\eta: A\{0\} \to R\Gamma(X, L) \tten_A R\Gamma(X, M)[n]
\]
is naturally a map in the filtered derived category with $A$ placed in filtered degree 0. 
In the case $M=L^*$, we will establish that in fact this map is a coevaluation in the filtered derived category $DF(A)$, in the sense that there is an evaluation morphism $$\varepsilon\colon R\Gamma(X,L)\tten_A R\Gamma(X,M)\to A\{0\}$$ which along with $\eta$ satisfies the axioms of dualisable objects.  
In other words, it is a coevaluation in the sense of symmetric monoidal categories. 

To do so, we will introduce an intermediate notion of a coevaluation which is in fact equivalent to the desired definition.

\begin{definition}
    Suppose $P,Q\in DF(A)$. Then a map $\eta \colon A\{0\}\to  P\underline{\otimes}^L Q$ is called a graded coevaluation if upon taking $\gr_0$, for each $i\in \ZZ$ the composite map
$$\gr_0(\eta)_i\colon A\to \bigoplus_{i\in \ZZ} \gr_iP\otimes^L \gr_{-i}Q\to    \gr_i P\otimes^L \gr_{-i}Q$$ are the coevaluation maps coming from a duality between $\gr_iP$ and $\gr_{-i}Q$ in $D(A)$ .
\end{definition}

\begin{proposition}
\label{lemma: copairing sym}
    Suppose $P$ and $Q$ are perfect complexes in $DF(A)$ and suppose that $\eta\colon A\{0\}\to P\underline{\otimes}^L_AQ$ is a graded coevaluation, then it is a coevaluation in $(DF(A),\otimes^L)$ in the sense of symmetric monoidal categories.
\end{proposition}
\begin{proof}
    Given the map $$\eta\colon A\{0\}\to P\underline{\otimes}^L_AQ$$ note that both $P$ and $Q$ are dualisable, so in fact admit duals given by $P^\vee:=R\underline{\mathrm{Hom}}(P,A\{0\})$ and $Q^\vee:=R\underline{\mathrm{Hom}}(Q,A\{0\})$. Now tensoring the map $\eta$ with $Q^\vee$ we get a map $$Q^\vee\to P\underline{\otimes}^L Q\underline{\otimes}^L Q^\vee\to P.$$
Taking the $\gr_i$ we see that this map is an isomorphism (in fact the identity) on the graded (by the hypothesis that $\eta$ is a graded copairing) and so must be an isomorphism in $DF(A)$. So we conclude that $$P\simeq R\underline{\mathrm{Hom}}(Q,A\{0\})=Q^\vee.$$

Now we observe that there is a canonical evaluation map $$\varepsilon\colon Q\underline{\otimes}^L R\underline{\mathrm{Hom}}(Q,A\{0\})\to A\{0\}$$ and so identifying $Q^\vee$ with $P$ we get a pairing map $$\epsilon\colon P\underline{\otimes}^L Q\to A\{0\} $$ which, by construction, satisfies the duality axioms.
 
\end{proof}

\begin{theorem}
    \label{thm: copairing} The map $\eta$ is a graded coevaluation with respect to the Serre duality pairing, and thus a coevaluation in $(DF(A),\otimes^L)$ in the sense of symmetric monoidal categories.
\end{theorem}
\begin{proof} 
Note that the functor $\mathrm{gr}$ is a symmetric monoidal functor, in particular it commutes with $\otimes^L$. 
Moreover $\mathrm{gr}_0R\Gamma(X,L)=R\Gamma(X,\E)$, $\mathrm{gr}_0R\Gamma(X,L)=R\Gamma(X,\E^*)$ and $\mathrm{gr}_{1}R\Gamma(X,L)=R\Gamma(X,\F)[-1]$, $\mathrm{gr}_{-1}R\Gamma(X,L)=R\Gamma(X,\F^*)[1]$.
Thus taking $\mathrm{gr}_0$ of the map $\eta$ we get the following morphism: $$A\xrightarrow{\gr_0{\eta}} R\Gamma(X,\E)\otimes_A^L R\Gamma(X,\E^*)\oplus R\Gamma(X,\F)\otimes_A^L R\Gamma(X,\F^*).$$ 
Projecting to the individual summands we obtain the maps $$\gr_0(\eta)_{\E} \colon A\to R\Gamma(X,\E)\otimes_A^L R\Gamma(X,\E^*)$$ and $$\gr_0(\eta)_{\F}\colon A\to R\Gamma(X,\F)\otimes_A^L R\Gamma(X,\F^*).$$

Since we can identify $\gr_0(\eta)_{\E}$ with $\tau_{\E}$ and similarly for $\F$, we see that this map is a filtered copairing compatible with Serre duality.  Finally, by Lemma \ref{lemma: copairing sym}, $\eta$ is a coevaluation in $(DF(A),\otimes^L)$ in the sense of symmetric monoidal categories. 
\end{proof}

\section{Duality of differential operators}
\label{sec: duality diff}
Our goal in this section is to use Theorem \ref{thm: copairing} to prove duality between differentials arising from the long exact sequences arising naturally from the objects in the context of that theorem.  
In particular, when $L = \mc{E}\rightarrow \mc{F}$ is taken to be the two-term complex $\O_X\xrightarrow{d}\Omega^1_{X/S}$, we obtain the desired duality between the two differentials in the Hodge-de Rham spectral sequence.  

\subsection{Dual distinguished triangles}

We begin with a lemma about dual distinguished triangles.

\begin{lemma}\label{dualtriangle}
    Let $P$ be a filtered perfect complex. Let $P^\vee:=R\underline{\mathrm{Hom}}(P,A)$ be the dual. Then the triangles $$\gr_{i+1}P\to \fil_i/\fil_{i+2}P\to \gr_{i}P$$ and 
$$\gr_{-i}P^\vee\to \fil_{-i-1}/\fil_{-i+1}P^\vee\to \gr_{-i-1}P^\vee$$ are dual in the category $D(A).$

\end{lemma}

\begin{proof}
 We need to check that apply $R\mathrm{Hom}(-,A)$ to the triangle $$\gr_{i+1}P\to \fil_i/\fil_{i+2}P\to \gr_{i}P$$ gives us the triangle $$\gr_{-i}P^\vee\to \fil_{-i-1}/\fil_{-i+1}P^\vee\to \gr_{-i-1}P^\vee.$$

This follows from the observation that $R\mathrm{Hom}(-,A)$ preserves triangles and the formulas $$\gr_iP^\vee=\gr_iR\underline{\mathrm{Hom}}(P,A\{0\})=R\mathrm{Hom}(\gr_{-i}P,A)$$ and $$\fil_{i}/\fil_{i+2}P^\vee=\fil_{i}/\fil_{i+2}R\underline{\mathrm{Hom}}(P,A\{0\})=R\mathrm{Hom}(\fil_{-i-1}/\fil_{-i+1}P,A).$$ 
\end{proof}

In the situation of the Lemma \ref{dualtriangle}, assume that $P$ has a two step filtration increasing filtration in filtered degrees $0$ and $1$. Then setting $i=0$ we get exact triangles $$\gr_1P\to P\to \gr_0P\xrightarrow{d_P} \gr_1P[1]$$ and the dual triangle 
$$\gr_{-1}P^\vee[-1]\xrightarrow{d_{P^\vee}}\gr_0P^\vee\to P^\vee\to \gr_{-1}P^\vee.$$

We will apply this to the case where $P$ is set to be $R\Gamma(X, \mc{E} \overset{D}\to \mc{F})$ with $\E$ in degree 0; in particular with $\mc{E}=\O_X$ and $\mc{F} = \Omega^1_{X/S}$.   In light of Theorem \ref{thm: copairing}, in this scenario we may identify $P^{\vee}$ with $R\Gamma(X, \mc{F}^* \overset{D^*}\to \mc{E}^*)[n]$ with $\E^*$ where $\mc{E}^*$ is in degree $0$.  Thus we obtain the dual distinguished triangles (after a shift) 

\[
R\Gamma(X, \mc{E}) \rightarrow R\Gamma(X, \mc{F})\rightarrow R\Gamma(X, L)\xrightarrow{+1} 
\]
and 
\[
 R\Gamma(X, \mc{F}^*)[n]\rightarrow R\Gamma(X, \mc{E}^*)[n]\rightarrow R\Gamma(X, L^*)[n] \xrightarrow{+1}. 
\]

\subsection{Proof of Lemma \ref{lemma: key}} 
The long exact sequence associated to the dual triangles above give differentials 
\begin{align*}
H^0(d^0)\colon H^0(X, \mc{E})\to H^0(X, \mc{F}), \qquad H^n(d^n)\colon H^{n}(X, \mc{F}^*)\to H^n(X, \mc{E}^*).
\end{align*} 

Since the triangles are dual, we would like to conclude that $H^0(d^0)$ is the dual of $H^n(d^n)$. 

\begin{lemma}
\label{lemma: tor}
    Assume $f\colon X\rightarrow S$ is a smooth and proper morphism of schemes with relative dimension $n$.  Then for any vector bundle $\E$ on $X$, the pushforward $R\Gamma(X, \E)$ has tor amplitude in $[0, n]$.
\end{lemma}
\begin{proof}
    By \cite[Proposition 23.146]{GortzWedhorn}, it suffices to show that for any $s\in S$ with residue field $k(s)$, we have $H^i(R\Gamma(X, \E)\otimes^L k(s)) = 0$ for all $i \notin [0,n]$. By tor-independent base-change \cite[\href{https://stacks.math.columbia.edu/tag/08IB}{Tag 08IB}]{stacks-project}, it follows that \[R\Gamma(X, \E) \otimes^L k(s) = R\Gamma(X_s, \E \otimes^L k(s)), \] where $X_s$ is the ($n$-dimensional) fibre over $s$. But $\E$ is locally free, hence $\E \otimes^L k(s) = \E|_s$. So in conclusion \[H^i(R\Gamma(X, \E)\otimes^L k(s)) = H^i(X_s, \E_s) = 0 \] 
    by Grothendieck's vanishing theorem.
    
\end{proof}

\begin{proposition}
\label{prop: dual noetherian}
    Assume $X\rightarrow S$ is smooth and proper with relative dimension $n$ and that $S=\Spec A$ is an affine scheme.  
    Then for any two-term complex of differential operators between vector bundles $L = \mc{E}\xrightarrow{d} \mc{F}$, the differentials
    $H^0(d^0) \colon H^0(X, \mc{E})\to H^0(X, \mc{F})$ is the dual of $H^n(d^n)\colon H^{n}(X, \mc{F}^*)\to H^n(X, \mc{E}^*)$.
\end{proposition} 

\begin{proof}
We have the following commutative diagram:
		\begin{equation*}
			\begin{tikzcd}[row sep = huge]
				R\Gamma(X, \mathcal{E}) \arrow[r, "\sim"] \arrow[d, "R\Gamma(d)"]& R\Hom(R\Gamma(X, \mc{E}^*), A)[-n]\arrow[d, "R\Gamma(d)'"]\\
				R\Gamma(X, \mc{F}) \arrow[r, "\sim"]\arrow[d] & R\Hom(R\Gamma(X, \mc{F}^*), A)[-n]\arrow[d]\\
				R\Gamma(X, L) \arrow[r, "\sim"] & R\Hom(R\Gamma(X, L^*), A)[-n],
			\end{tikzcd}
		\end{equation*}
		where the top two rightward quasi-isomorphisms are given by Serre duality, and the columns are both triangles. Taking $H^0$, we find: 
		\begin{equation*}
			\begin{tikzcd}[row sep = huge]
				H^0(X, \mathcal{E}) \arrow[r, "\sim"] \arrow[d, "H^0(d^0)"]& \Ext^{-n}(R\Gamma(X, \mc{E}^*), A)\arrow[d, "d'"]\\
				H^0(X, \mc{F}) \arrow[r, "\sim"]\arrow[d] & \Ext^{-n}(R\Gamma(X, \mc{F}^*), A)\arrow[d]\\
				H^0(R\Gamma(X, L)) \arrow[r, "\sim"] & \Ext^{-n}(R\Gamma(X, L^*), A),
			\end{tikzcd}
		\end{equation*}
		We wish to identify $d'$ with \[\Hom(H^n(X, \mc{E}^*), A) \xrightarrow{H^n(d^n)^*} \Hom(H^n(X, \mc{F}^*), A). \]  
		
		Since $X\rightarrow S$ is smooth and proper of dimension $n$, by Lemma \ref{lemma: tor}, $R\Gamma(X, \E)$ has tor amplitude in $[0, n]$. 
  Then by \cite[\href{https://stacks.math.columbia.edu/tag/0654}{Tag 0654}]{stacks-project} 
  we have a projective resolution of $R\Gamma(\E)$  \[R\Gamma(\E) \xrightarrow{\sim} [0\rightarrow P^0 \rightarrow...\rightarrow P^{n-1}\rightarrow P^n\rightarrow 0].  \] 
  Note that we have the exact sequence: 
		\begin{equation}\label{equation:right-exact-sequence}
			P^{n-1}\rightarrow P^n \rightarrow H^n(X, \E)\rightarrow 0.
		\end{equation}
		Now we compute $R\Hom(R\Gamma(\E), A)[-n]$:
		\[R\Hom(R\Gamma(\E), A)[-n] = \Hom(P^\bullet, A)[-n] = [0 \rightarrow \Hom(P^n, A) \rightarrow \Hom(P^{n-1}, A) \rightarrow...]. \] Thus $\Ext^{-n}(R\Gamma(\E), A) = \ker (\Hom(P^n, A) \rightarrow \Hom(P^{n-1}, A))$. 
  But by (\ref{equation:right-exact-sequence}) and left-exactness of $\Hom$, it follows \[\ker (\Hom(P^n, A) \rightarrow \Hom(P^{n-1}, A)) = \Hom(H^n(X, \E), A). \] 
  Applying this to the cases $\E = \Omega^{n-1}_{X/S}$ and $\E = \Omega^n_{X/S}$, we find that 
  \[
  \Ext^{-n}(R\Gamma(X, \E^*), A) =\Hom(H^n(X, \E^*), A) 
  \] 
  and 
  \[
  \Ext^{-n}(R\Gamma(X, \mc{F}^*), A) =\Hom(H^n(X, \mc{F}^*), A). 
  \] A simple diagram chase shows that the maps $d'$ and $H^n(d^n)^*$ are the same. 
    
\end{proof}

\begin{corollary}
(Alternative proof of \cite[\href{https://stacks.math.columbia.edu/tag/0G8J}{Tag 0G8J}]{stacks-project})
\label{lemma: key noetherian}
    Let $S$ be a quasi-compact and quasi-separated scheme and let $f\colon X\rightarrow S$ be a proper smooth morphism of schemes all of whose fibres are nonempty and equidimensional of dimension $n$.  Then the map $d^n\colon R^nf_*\Omega^{n-1}_{X/S}\rightarrow R^nf_*\Omega^n_{X/S}$ is zero.
\end{corollary}
\begin{proof}
    The statement is local on $S$; therefore it suffices to take $S=\Spec A$ affine. 
 Applying Proposition \ref{prop: dual noetherian} to $L=\O_X\rightarrow \Omega^1_{X/S}$, we may identify $(d^n)^*\colon \Hom(H^n(X, \Omega^n_{X/S}), A) \rightarrow \Hom(H^n(X, \Omega^{n-1}_{X/S}), A)$ with $d^0\colon f_*\O_X\rightarrow f_*\Omega^1_{X/S}$, which by \cite[\href{https://stacks.math.columbia.edu/tag/0G8F}{Tag 0G8H}]{stacks-project} is 0.  Since $H^n(X, \Omega^n_{X/S})$ is free over $A$, this implies that $d^n$ itself is 0, as desired.
\end{proof}

\bibliographystyle{alpha}
\bibliography{main}
\end{document}